%% file: StratiRCD.tex
\documentclass[a4paper,reqno]{amsart}
\usepackage[T1]{fontenc}
\usepackage[utf8x]{inputenc}
\usepackage[english]{babel}
\usepackage{amsmath,amsfonts,amssymb}
\usepackage{esint}  
\usepackage{graphicx}
\usepackage{hyperref}
\usepackage{enumerate}
\usepackage{xcolor}

\usepackage{amsthm}

\theoremstyle{plain}
\newtheorem*{D*}{Definition}
\newtheorem*{theorem*}{Theorem}
\newtheorem*{prop*}{Proposition}
\newtheorem*{coro*}{Corollary}

\newtheorem{cor}{Corollary}[section]
\newtheorem{theorem}[cor]{Theorem}
\newtheorem{lemma}[cor]{Lemma}
\newtheorem{proposition}[cor]{Proposition}

\theoremstyle{definition}
\newtheorem{definition}[cor]{Definition}
\newtheorem{remark}[cor]{Remark}
\newtheorem{thmA}{Theorem}

\newtheorem{propA}{Proposition}

\newtheorem{LemmaA}[propA]{Lemma}

\newtheorem{corA}[thmA]{Corollary}
\newcommand{\R}{\ensuremath{\mathbb R}}
\newcommand{\N}{\ensuremath{\mathbb N}}

\newcommand{\C}{\ensuremath{\mathbb C}}
\newcommand{\s}{\ensuremath{\mathbb S}}
\newcommand{\proj}{\ensuremath{\mathbb P}}
\newcommand{\eps}{\ensuremath{\varepsilon}}

\newcommand{\distArg}[1]{\ensuremath{\mbox{d}_{#1}}}
\newcommand{\norm}[1]{\ensuremath{\left\|#1\right\|}}
\newcommand{\cutoff}{\ensuremath{\rho_{\varepsilon}}}

\DeclareMathOperator{\vol}{Vol}
\DeclareMathOperator{\ric}{Ric}
\DeclareMathOperator{\diam}{diam}
\DeclareMathOperator{\Lip}{Lip}
\DeclareMathOperator{\lip}{Lip}
\DeclareMathOperator{\m}{m}
\DeclareMathOperator{\Ch}{Ch}
\DeclareMathOperator{\CBB}{CBB}
\newcommand{\ke}{K}
\newcommand{\Lo}{\mathcal{L}}
\usepackage{xcolor}

\usepackage{xcolor}

\renewcommand{\S}{\mathbb{S}^{n-1}}

\title{Stratified spaces and synthetic Ricci curvature bounds}

\author{J.~Bertrand}
\address{Université P. Sabatier, IMT}
\email{bertrand@math.univ-toulouse.fr}
\author{C.~Ketterer }
\address{University of Toronto}
\email{ckettere@math.toronto.edu}
\author{I.~Mondello}
\address{Université de Paris-Est Créteil, Laboratoire d'Analyse et mathématiques appliquées}
\email{ilaria.mondello@u-pec.fr}
\author{T.~Richard}
\address{Université de Paris-Est Créteil, Laboratoire d'Analyse et mathématiques appliquées}
\email{thomas.richard@u-pec.fr}
\date{}

\begin{document}

\nocite{*}

\begin{abstract}
We prove that a compact stratified space satisfies the Riemannian
curvature-dimension condition $RCD(K,N)$ if and only if its Ricci
tensor is bounded below by $K \in \R$ on the regular set, the cone
angle along the stratum of codimension two is smaller than or equal to
$2\pi$ and its dimension is at most equal to $N$. This gives a new
wide class of geometric examples of metric measure spaces satisfying
the $RCD(K,N)$ curvature-dimension condition, including for instance
spherical suspensions, orbifolds, Kähler-Einstein manifolds with a
divisor, Einstein manifolds with conical singularities along a
curve. We also obtain new analytic and geometric results on stratified
spaces, such as Bishop-Gromov volume inequality, Laplacian comparison,
Lévy-Gromov isoperimetric inequality. Our result also implies a similar characterization of compact stratified spaces carrying a lower curvature bound in the sense of Alexandrov.
\end{abstract}
\maketitle

\input{intro}
\input{strati}

\input{RCD}

\section{Examples of stratified spaces with a singular Ricci lower bound}

\input{examp}

\input{geometric}

\input{mainproof}


\input{appendix}

\bibliography{biblioStratifiedRCD}
\bibliographystyle{amsalpha}

\end{document}

%% file: intro.tex
\section*{Introduction}


Singular metric spaces naturally appear in differential geometry when considering quotients of smooth manifolds, their Gromov-Hausdorff limits, when they exist, or geometric flows. One of the main questions when dealing with singularities, is how to define a good notion of curvature, or of curvature bounds. One of the possible and more efficient ways to answer this question is given by the work of K.-T.~Sturm \cite{stugeo1} \cite{stugeo2}, and of J.~Lott together with C.~Villani \cite{lottvillani}, which initiated the study of synthetic Ricci curvature bounds on metric measure spaces. In the recent years, such study has given rise to a rich theory where significant analytic and geometric results intertwine. The idea for the $CD(K,N)$ curvature-dimension condition is to define a lower bound $K$ for the curvature, and an upper bound $N$ for the dimension, in terms of convexity for entropy functionals in the appropriate space of probability measures, the $L^2$-Wasserstein space. L.~Ambrosio, N.~Gigli and G.~Savaré \cite{agsriemannian} \cite{giglistructure} refined the previous condition and introduced the \emph{Riemannian curvature-dimension condition} $RCD(K,N)$, which rules out Finsler geometries. 

Some of the many good features of the Riemannian curvature-dimension condition is that it corresponds, in the setting of smooth Riemannian manifolds, to a standard lower Ricci bound, and moreover it is stable under measured Gromov-Hausdorff convergence ($mGH$ convergence for short). Therefore, $mGH$-limits of smooth manifolds whose Ricci curvature is uniformly bounded below are the first, possibly singular, examples of metric measure spaces satisfying the $RCD$ condition. Other examples are given by finite dimensional Alexandrov spaces with a lower curvature bound \cite{kuwae, Petrunin12}, and weighted manifolds with Bakry-Émery tensor bounded below. 
In more general terms, it is now known that all the constructions which preserve a lower Ricci bound in the context of smooth manifolds, cones, suspensions, quotients, (metric) foliations/submersions, also preserve, under some technical assumptions, the $RCD$ condition in the setting of metric measure spaces  \cite{KD,KWarped,GalazKellMondinoSosa}. 
However, all these examples are, in some sense, rigid: cones in the work of the second author carry an exact cone metric; an orbifold singularity is modeled on a cone over a quotient of the sphere, and other cone sections are not allowed. If we consider a more general and flexible model for conical singularities, isolated or not, on a smooth Riemannian manifold, there isn't any known geometric criterion to establish whether a synthetic lower Ricci bound holds. 

The aim of this paper is to fill this gap and present a new class of geometric examples satisfying a  $RCD$ condition, which includes in particular orbifolds, spherical suspensions over smooth manifolds, and manifolds with conical singularities (isolated or not). More precisely, we give a criterion on compact \emph{stratified spaces}, as defined in the works of K.~Akutagawa, G.~Carron, R.~Mazzeo \cite{ACM12} and of the third author, under which such metric measure spaces satisfy the Riemannian curvature-dimension condition. 

Stratified spaces can be seen as a generalization of manifolds with isolated conical singularities; in fact, they can be decomposed into a regular set $X^{\tiny{reg}}$, which is a smooth manifold of dimension, and a closed singular set, made of \emph{singular strata}  of possibly different dimensions, with a local ``cone-like'' structure. This means that a tubular neighbourhood of a singular stratum is the product of an Euclidean ball and a cone, thus we can consider not only isolated conical singularities, but also conical singularities along a curve or more generally along a submanifold. We focus our attention on compact stratified spaces without boundary, hence the minimal codimension of a singular stratum is two. 

Stratified spaces were first introduced in topology by H.~Withney and R.~Thom, then later studied from a more analytical point of view starting from the work of J.~Cheeger \cite{Cheeger83}. In this paper, we consider stratified spaces with a Riemannian approach; indeed, it is possible to define an iterated edge metric (see  \cite{ALMP} , \cite{ACM12}) which is a Riemannian metric on the regular set, and whose asymptotic expansion is \emph{close} to a model metric, depending on the strata to which the point where the expansion is performed belongs. The fact that we only require closeness to a model geometry gives more flexibility about the choice of the iterated edge metric, including its regularity.

In \cite{M14,M15}, the third author studied compact stratified spaces with a lower Ricci curvature bound. Note that the Ricci tensor is only well-defined on the regular set of a stratified space; one has to be careful about the behaviour of the metric near singular strata, and in particular near the stratum of codimension two. Indeed, the singularities along this stratum are modeled on a two-dimensional metric cone, which has an angle. If such angle is smaller than $2\pi$, then the cone has positive curvature in the sense of Alexandrov, negative otherwise. This plays an important role in the following definition: 

\begin{D*}[Singular lower Ricci bound]
Let $X$ be a compact stratified space of dimension $n$ endowed with an iterated edge metric $g$. Let $K \in \R$. We say  that $g$ has {\it singular} Ricci curvature bounded from below by $K$ if
\begin{itemize}
\smallskip
 \item[(i)] $\ric_g\geq K$ on the regular set $X^{\tiny{reg}}$,
 \medskip
 \item[(ii)] the angle $\alpha$ along the stratum $\Sigma^{n-2}$ is smaller than or equal to $2\pi$. 
\end{itemize}
\end{D*}

Observe that we do not need to give any condition on the strata of codimension larger than two, since the condition on the regular set suffices to control the behaviour of the Ricci curvature of the cone sections at those strata. It is not the case for the codimension-two stratum.  
Using the stability under $mGH$ convergence of the $RCD$ condition, one can guess that some assumption on the cone angle is needed for a $RCD$ condition to hold on a stratified space: in fact, if the space is $RCD$, then all the tangent cones must have a non-negative $RCD$ curvature bound, and K.~Bacher and K-T.~Sturm \cite{bastco} proved that a cone over a manifold of diameter larger than $\pi$ does \emph{not} satisfy a $CD$ condition.

In dimension two, the previous definition corresponds to surfaces with sectional curvature bounded below and isolated conical singularities with angles smaller than $2\pi$. Such singular surfaces are known to be Alexandrov spaces \cite{alexandrov}, and thus are examples of $RCD$ spaces. In higher dimension, more general singularities can occur; stratified spaces satisfying the previous definition include orbifolds, Kähler-Einstein manifolds with a divisor, spherical suspensions over smooth manifolds (or stratified spaces) with a lower Ricci bound.

As proven by the second and the third author, both $RCD$ and stratified spaces share properties with smooth Riemannian manifolds  involving the bottom of the spectrum or the diameter; corresponding rigidity results also hold \cite{M14,M15,KD, KObata}.
It is then natural to expect, but not elementary to prove, that stratified spaces with a singular lower Ricci bound also satisfy a $RCD$ condition. We are going to prove that the former condition is actually equivalent to the latter. More precisely, taken for granted  that a compact stratified space admits a natural distance $d_g$ as well as a volume measure $v_g$ (see Section 1 for more on these points), our main theorem  states the following: 

\begin{thmA}
\label{MainThm}
Let $(X,g)$ be a compact stratified space  endowed with an iterated edge metric $g$. Equipped with its natural distance $d_g$ and measure $v_g$, the stratified space $(X,d_g,v_g)$ satisfies the $RCD(K,N)$ condition if and only if its dimension is smaller than or equal to $N$ and the iterated edge metric $g$ has singular Ricci curvature bounded below by $K$. 
\end{thmA}

Under the assumption of a singular lower Ricci bound, we actually prove a condition referred to as $BE(K,N)$, which is known to be equivalent to $RCD(K,N)$ under some conditions (\cite{EKS,AMS15}). The condition $BE(K,N)$ stays for Bakry-Émery and is inspired by the $\Gamma_2$-calculus developed by these authors, built on the Bochner formula. We emphasize that the proof of the $BE(K,N)$ condition in our setting relies on a non trivial regularity result for the eigenfunctions of the Laplacian due to \cite{M14}, which strongly depends on the angle $\alpha$ along $\Sigma^{n-2}$ being smaller than $2\pi$. To prove the reverse implication, we only need stability properties of the $RCD$ condition mentioned above.

Not only the previous theorem gives a new ample class of geometric examples of $RCD(K,N)$ spaces, but also allows us to apply the rich theory of $RCD(K,N)$ spaces to stratified spaces. As a consequence, we obtain previously unknown results in this setting such as Laplacian comparisons, Bishop-Gromov volume estimate, Lévy-Gromov isoperimetric inequality. Note that it is not immediate to deduce Laplacian comparisons and volume estimates on stratified spaces, since the classical proofs require regularity properties of the distance function, that can fail to be true when considering the distance to a singular point. Let us also add that it is a difficult problem to understand the behaviour of (long) geodesics on stratified spaces; for instance, it is not known whether the regular set of a compact stratified space with a singular lower Ricci bound is geodesically convex. 

Nevertheless, our main theorem implies that any stratified space $(X,g)$ with a singular lower Ricci bound is \emph{essentially non-branching} and, as it was pointed to us by V. Kapovitch, its regular set $X^{reg}$ is \emph{almost everywhere convex}. By applying a result of N.~Li \cite{NanLi}, this weak level of control on the geodesics turns out to be enough to show an analogue of Theorem \ref{MainThm} for lower bounds on the sectional curvature:
\begin{corA}\label{MainCor}
Let $(X,g)$ be a compact stratified space. Then $(X,d_g)$ has
curvature bounded from below  by $k$ in the sense of Alexandrov if and
only if the following two conditions are satisfied:
\begin{enumerate}[(i)]
\item The sectional curvature of $g$ is larger than or equal to $k$ on $X^{reg}$.
\item The angle $\alpha$ along the singular stratum is at most $2\pi$.
\end{enumerate}
\end{corA}

Finally, we would like to point out that $RCD(K,N)$ spaces includes, but are not necessarily, Ricci limit spaces. For example, it is known that the spherical suspension over $\R\proj^2$ is a $RCD(K,N)$ space, and it is a compact stratified space with singular Ricci lower bound as well, but, as observed by G.~De Philippis, A.~Mondino and P.~Topping, it cannot be a non-collapsed limit of Riemannian manifolds. Indeed, M.~Simon proved in \cite{Simon2012} that the Gromov-Hausdorff limit of a sequence of 3-manifolds with a lower bound on the Ricci tensor and an upper bound on the diameter must be a topological manifold, which the spherical suspension over $\R\proj^2$ is not. M.~Simon's results proves a conjecture of of M.~Anderson, J.~Cheeger, T.-H.~Colding and G.~Tian, which has also been shown in \cite{SimonTopping17} without assuming the upper bound on the diameter. 
It is in general a very difficult question to find new examples of $RCD$ spaces not arising as Gromov-Hausdorff limits of smooth manifolds; moreover, even in the simple case of cones and spherical suspension, for example over $\R\proj^2$, it is not easy to figure out whether they are \emph{collapsed} limits of Riemannian manifolds or not. Having a wider class of geometric examples of $RCD$ spaces could contribute to make progresses in solving this question. 

The paper is organized as follows. The first section is devoted to illustrate some notions about stratified spaces which will be used throughout the paper. In the second section, we recall the basics of curvature-dimension conditions on metric measure spaces. We then give some examples of stratified spaces which carry a singular lower Ricci bound, and thus, thanks to our main theorem, satisfy the $RCD(K,N)$ condition. The fourth section is devoted to reformulating some analytic and geometric results about $RCD(K,N)$ spaces in the setting of stratified spaces with a singular lower Ricci bound, as an application of our main theorem. It also includes the proof of Corollary \ref{MainCor}. The last section contains the proof of the main theorem. \\

\noindent \textbf{Acknowledgments:} We would like to thank V.~Kapovitch for the proof of Corollary B, N. Gigli, A.~Mondino, V.~Bour and R.~Perales for useful discussions and comments on the preliminary version of this article. The third author is supported by a public grant overseen by the Centre National de Recherche Scientifique (CNRS) as part of the program ``Projet Exploratoire Premier Soutien Jeunes Chercheur-e-s''. . 

%% file: strati.tex
\section{Preliminaries on stratified spaces}\label{sec-pre-stra}


\subsection{Definition of compact stratified spaces and iterated edge metrics}

\subsubsection{Definition and differential properties}

The topological definition of a stratified space can be given by induction with respect to a quantity called \emph{depth} of the space. For the sake of simplicity, we present here a definition by induction on the dimension. In the definition and in the following we will use \emph{truncated cones}. A truncated cone over a metric space $Z$ is the quotient space $([0,1]\times Z)/\sim$ with the equivalence relation $(0,z_1)\sim (0,z_2)$ for all $z_1,z_2\in Z$. If the interval is not $[0,1]$ but $[0,\delta)$ for some $\delta >0$ we will write $C_{[0,\delta)}(Z)$.

\begin{definition}
A one-dimensional compact stratified space is simply a connected compact differentiable manifold of dimension one. For $n$ larger than one, assume that we have defined $(n-1)$ dimensional compact stratified spaces. Then, an $n$-dimensional stratified space is a connected compact topological space $X$ such that the following properties hold:
\begin{itemize} 
\medskip
 \item[(a)]There exists a decomposition of $X$ in \emph{strata}
$$X = {\bigsqcup}_{j=0}^n\ \Sigma^j,$$ 
where $\Sigma^0$ is a finite set of points, and $\Sigma^j$ are smooth, possibly open, manifolds of dimension $j\in \{1,\cdots ,n\}$. Each $\Sigma^j $ is called a {\it stratum}. We assume $X$ is without boundary, namely $\Sigma^{n-1}= \emptyset$. The closure of $\Sigma^j$ is required to satisfy
\begin{equation}\label{clos}
 \displaystyle \overline{\Sigma^j} \subset \bigcup_{l \leq j} \Sigma^l.
\end{equation}

\noindent We further define the {\it regular set} $X^{reg}$ as the stratum $\Sigma^n$ of highest dimension and the {\it singular set} $\Sigma$ as its complement, namely  $\Sigma= \cup_{j=0}^{n-2}\Sigma^j$. The strata of dimension $j \leq (n-2)$ are called \emph{singular strata}. Thanks to (\ref{clos}) $\Sigma$ is a closed set, thus the regular set $X^{reg}$ is an {\it open and dense} subset of $X$.

\medskip
\item[(b)] Each connected component $\widetilde{\Sigma}^j$ of the singular stratum $\Sigma^j$ of dimension $j \leq (n-2)$ admits a neighbourhood $\mathcal{U}_j$ homeomorphic to the total space of a bundle of truncated cones over $\widetilde{\Sigma}^j$. More precisely, there exists a retraction 
$$\pi_j: \mathcal{U}_j \rightarrow \widetilde{\Sigma}^j,$$ and a compact {\it stratified} space $Z_j$ of dimension $(n-j-1)$ such that $\pi_j$ is a cone bundle whose fibre is a truncated cone over $Z_j$. We set $\rho_j: \mathcal{U}_j \rightarrow [0,1]$ the {\it radial function} where $ \rho_j(x)$ stands for the radial factor in the conical fiber $\pi_j^{-1}(\{x\})$.  The stratified space $Z_j$ is called the \emph{link} of the stratum.
\end{itemize}
\end{definition}

\noindent 
We are interested in studying \emph{smoothly} stratified spaces. This means that the cone bundle given in the definition is assumed to satisfy a {\it smoothness} property that we now describe. Indeed, we have a notion of \emph{local chart} in a neighbourhood of a singular point and such chart is smooth on the regular subset of the neighbourhood. More precisely, for each $x \in \Sigma^j$ there exists a relatively open ball $B^j(x) \subset \Sigma^j$ and a homeomorphism $\varphi_x$ such that: 
\begin{center}
$\varphi_x: B^j(x) \times C_{[0,\delta_x)}(Z_j) \longrightarrow \mathcal{W}_x:=\pi_j^{-1}(B^j(x))$, 
\end{center}
satisfies $\pi_j\circ \varphi_x=p_1$ where $\delta_x>0$, and $p_1$ is the projection on the first factor of the product \mbox{$B^j(x)\times C_{[0,\delta_x)}(Z_j)$}. 
Moreover, $\varphi_x$ restricts to a smooth {\it diffeomorphism} on the regular sets, that is from $\left(B^j(x)\times C_{[0,\delta_x)}(Z_j^{reg}\right))\setminus \left(B^j(x))\times \{0\}\right)$ onto the regular subset $\mathcal{W}_x\cap X^{reg}$ of $\mathcal{W}_x$.  


\begin{remark}
First examples of stratified spaces are manifolds with isolated conical singularities and orbifolds. In this second case, the links are quotients of the sphere by a finite subgroup of $O(n)$, acting freely on $\R^n \setminus \{0\}$. Note that not all stratified spaces are necessarily orbifolds. 
\end{remark}

\begin{remark}
To have a better picture of the local geometry, let us point out the case of a stratum $\Sigma^{n-2}$ of minimal codimension: each point of $\Sigma^{n-2}$ has a neighbourhood which is homeomorphic to the product of a ball in $\R^{n-2}$ and a two-dimensional truncated cone. There is only one possibility for the link of $\Sigma^{n-2}$: since it has to be a one-dimensional compact stratified space, it is a circle $\s^1$. 
\end{remark}




\subsubsection{Iterated edge metric}
We are going to define a class of Riemannian metrics on a stratified space: iterated edge metrics. Such metrics are proven to exist in \cite{ALMP}. We briefly sketch the inductive construction of an iterated edge metric. Recall that in dimension one a stratified space is a standard smooth manifold; in this case an iterated edge metric is nothing but a smooth Riemannian metric. Assume that we have constructed an admissible iterated edge metric on compact stratified spaces of dimension $k \leq (n-1)$ and consider a stratified space $X^n$. In order to define an iterated edge metric, we first set a model metric on $\mathcal{U}_j$, the neighbourhood of a connected component of the singular stratum $\Sigma^j$ introduced in the previous paragraph.


Consider $k_j$ a symmetric 2-tensor on $\partial \mathcal{U}_j=\rho_j^{-1}(1)$ which restricts to an admissible metric on each fibre of the cone bundle $\pi_j$ and vanishes on a $j$-dimensional subspace. Such tensor exists because the link $Z_j$ is a stratified space of dimension smaller than $n$. We define the \emph{model metric} on $\mathcal{U}_j$ as follows:
$$g_{0,j}=\pi_j^*h+d\rho_j^2+\rho_j^2k_j,$$
where $h$ is a smooth Riemannian metric on $\Sigma^j$. Observe that, in terms of the local coordinates given by a chart $\varphi_x$ around $x \in \Sigma_j$, if $(y,\rho_j,z)$ are the coordinates of a point in $\mathcal{U}_j$ with $y$ in $\R^j$ and $z$ in $Z_j$, $h$ only depends on $y$, while $k$ depends on $y$ and $z$. 

\begin{definition}[Iterated edge metric]
Let $X$ be a stratified space with strata $\Sigma^j$ and links $Z_j$; let $g_{0,j}$ be the model metric defined above. A smooth Riemannian metric $g$ on the regular set $X^{\tiny{reg}}$ is said to be an \emph{iterated edge metric} if there exist constants $\alpha, \Lambda >0$ such that for each $j$ and for each $x\in \Sigma^j$ we have: 
\begin{equation}\label{e-def-ite-met}
 |\varphi_x^*g-g_{0,j}|\leq \Lambda r^{\alpha}, \mbox{ on } \mathbb{B}^j(r)\times C_{(0,r)}(Z_j^{\tiny{reg}}),
\end{equation}
for any $r < \delta_x$ (where $\delta_x$ and the local chart $\varphi_x$ are defined as above and $|\cdot|$ refers to the norm on tensors induced by $g_{0,j}$). 
\end{definition}

\begin{remark}
When we consider a stratum $\Sigma^{n-2}$ of codimension 2, the link is a circle $\s^1$, and therefore the model metric around $x \in \Sigma^{n-2}$ has the following form: 
$$g_{0,n-2}= \pi_{n-2}^*h+d\rho_{n-2}^2+\rho_{n-2}^2(a^2_xd\theta^2),$$
where $a^2_xd\theta^2$ is a metric on $\s^1$, for $a_x\in (0, +\infty)$. We refer to $\alpha_x=a_x \cdot 2\pi$ as the \emph{angle} of $\Sigma^{n-2}$ at $x$, since $\alpha_x$ is the angle of the exact cone $(C(\s^1), d\rho^2+a_x^2\rho^2d\theta^2)$. Note that $\alpha_x$ may depend on $x$, and it can be smaller or larger than $2\pi$. This will play a crucial role in studying  lower curvature bounds. 
\end{remark}

\subsection{Stratified space viewed as Metric Measure Space}

In this part, we introduce a distance and a measure on a compact stratified space equipped with an iterated edge metric. To this aim, we shall follow and use properties from the book \cite{bbi}.

\subsubsection{Distance on a stratified space}\label{dist_strat}

\begin{definition}[Length structure]
If $g$ is an iterated edge metric for $X$, one introduces the associated length structure as follows. A continuous curve $\gamma:[a,b]\rightarrow X$ is said to be {\it admissible} if the image $\gamma([a,b])$ is contained in $X^{reg}$ up to finitely many points; we further assume $\gamma$ to be $C^1$ on the complement of this finite set. We then define the length of such a curve as
\begin{align*}
\mbox{L}_g(\gamma)=\int_a^b|\dot{\gamma}(t)|dt.
\end{align*}
Consequently, we define the distance between two points $x,y \in X$ as follows:
$$d_{g}(x,y)=\inf\{\mbox{L}_g(\gamma)| \gamma: [a,b] \rightarrow X \mbox{ admissible curve s.t. } \gamma(a)=x, \gamma(b)=y\}.$$ 
\end{definition}

It is rather straightforward to check that the above length structure meets the hypotheses described in \cite[Chapter 2]{bbi} which ensures that $d_g$ is indeed a distance.

We shall also use the following result: 

\begin{lemma}
\label{PropLipCurve}
Let $(X,g)$ be a compact stratified space of dimension $n$ endowed with the iterated edge metric $g$. Let $\gamma:[0,1] \rightarrow X$ be an admissible curve. For any $\varepsilon>0$, there exists an admissible curve $\gamma_{\varepsilon}$ with the same endpoints as $\gamma$ and such that $\gamma_{\varepsilon}((0,1))$ is contained in the regular set $X^{\tiny{reg}}$. Moreover
$$\mbox{L}_g(\gamma_{\varepsilon}) \leq \mbox{L}_g(\gamma)+ \varepsilon.$$
\end{lemma} 

More details about the length structure, including a proof of the above Lemma, can be found in the appendix.

As proved in \cite[Chapter 2]{bbi}, the length structure induced by $d_g$ gives rise to a distance $\hat{d}$ which coincides with $d_g$. This fact allows us to consider the larger set of {\it rectifiable curves} (where the length is intended as the standard one  on a metric space). Note also that according to the above lemma, the distance $d_{g}(x,y)$ is the infimum of lengths of admissible curves whose range is in $X^{reg}$ except maybe the endpoints.

Consequently, $(X,d_g)$ is a compact length space, thus Ascoli-Arzela's theorem implies it is a {\it geodesic space}.


\vspace{0.3cm}

\subsubsection{Tangent cones and geometry of small geodesic balls}\label{subsubsec:tangentcones}

The local geometry of a singular point $x \in \Sigma$ is also well-understood. Let us start with a result on tangent cones, namely the Gromov-Hausdorff limits of the pointed metric spaces $(X, \varepsilon^{-1}d_g, x)$ as $\varepsilon$ goes to zero.
\begin{lemma} Let $x$ be a point in $(X, d_g)$. Then there exists a {\it unique} tangent cone $T_x X$ at $x$. Moreover, when $x \in \Sigma^j$, this cone is isometric to $\R^j \times C(Z_j) $ equipped with the (distance induced by the) product metric $dx^2+d\rho^2+\rho^2k_j $. 
\end{lemma}
\begin{remark}
By definition of an iterated edge metric, the tangent cone at a regular point is isometric to Euclidean space $\R^n$.

For $x\in \Sigma^j$, a change of variables proves that the tangent cone at a singular point is isometric to
$$(C(S_x), dr^2+r^2h_x)$$
where $S_x$ is the $(j-1)$-fold spherical suspension of the link $Z_j$: 
$$\left( \left[0, \frac{\pi}{2} \right]\times \s^{j-1} \times Z_j, d\varphi^2+\cos^2\varphi g_{\s^{j-1}}+\sin^2\varphi k_j\right).$$
We refer to $S_x$ as the \emph{tangent sphere} at $x$.  Note that $S_x$ is a compact stratified space of dimension $(n-1)$.

Note also that \ref{e-def-ite-met} implies

\begin{equation}\label{e-control-ite}
 |\varphi_x^*g-g_{0,j}|\leq \Lambda r^{\alpha}, \mbox{ on } B_{0,j} (x,r) \cap X^{\tiny{reg}} \subset \mathbb{B}^j(r)\times C_{(0,r)}(Z_j^{\tiny{reg}})
\end{equation}
the open ball of radius $r$ w.r.t. to the model metric $g_{0,j}$ centered at $x$. Note that this ball is homeomorphic to $C_{[0,r)} (S_x)$.
\end{remark}

\noindent Rescaling the distance by a factor $1/\varepsilon$ around a point $x$ amounts to rescale the iterated edge metric by a factor $1/\varepsilon^2$. By definition of this metric around a singular point, it is then  not difficult to show existence and uniqueness of the tangent cone at any such point $x$ of a stratified space, see for example Section 2.1 in \cite{ACM12} for more details about the above properties.

\vspace{0.3cm}

\noindent The tangent sphere allows us to give a different and  useful description of geodesic balls around a point; the idea is that a geodesic ball around a singular point is ``not far'' from being a cone over its tangent sphere. We refer to Section 2.2 of \cite{ACM14} for proofs of the following properties.

Namely, for any $x \in \Sigma$, there exists a sufficiently small radius $\varepsilon_x$, a positive constant $\kappa$ and an open set $\Omega_x$ satisfying: 
\begin{itemize}
\item the geodesic ball $B(x,\varepsilon_x)$ is contained in $\Omega_x$, 
\item $\Omega_x$ is homeomorphic to the truncated cone $C_{[0,\kappa \varepsilon_x)}(S_x)$, and the homeomorphism $\psi_x$ sends the regular part of $\Omega_x$ to the regular set of $C_{[0,\kappa \varepsilon_x)}(S_x)$;
\item on the regular part of the ball $B(x,\varepsilon_x) \cap X^{\tiny{reg}}$ we control the difference between $g$ and the exact cone metric $g_C=dr^2+r^2h_x$: 
\begin{equation}
\label{g_C}
|\psi_x^*g-(dr^2+r^2h_x)|\leq \Lambda\varepsilon_x^{\alpha},
\end{equation}
where $\Lambda$ is a positive constant, and $\alpha$ is the same exponent appearing in the definition of the iterated edge metric $g$. 
\end{itemize} 
 This implies a similar estimate for the distance functions. More precisely, on the cone over $(S_x,h_x)$, let us  consider the exact cone distance: 
$$\distArg{C}((t,y),(s,z))=\sqrt{t^2+s^2-2st \cos (d_{h_x}(y,z) \wedge \pi}),$$
where $(t,y)$, $(s,z)$ belong to $C(S_x)$ and $a\wedge \pi$ is the minimum between $a$ and $\pi$.

Now, given $x \in \Sigma$ and a radius $0< \varepsilon < \varepsilon_x$, we have, for any point $y$ in $B(x,\varepsilon)$ with coordinates $(r,z)$ in $C_{[0,\kappa \varepsilon_x)}(S_x)$, the estimate: 
\begin{equation}
\label{d_C}
|d_g(x,y) - \distArg{C}(0, (r,z))| \leq \Lambda \varepsilon^{\alpha+1}.
\end{equation}
In other terms, the distance function from a point $x$ in the singular set is not far  from being the exact cone distance in $C(S_x)$ in small geodesic balls centred at $x$.
However, it is still quite difficult to say anything about the local behaviour of geodesic at singularities.

Observe that, thanks to the compactness of the stratified space, we can choose a uniform $\eps_0$ such that for any $x \in X$ the ball $B(x,\eps_0)$ satisfies the previous properties.


\subsubsection{Measure on a stratified space}
\label{volume-measure}
We end this part with the definition of the {\it volume measure} on $X$. We then show that the volume shares properties with the standard Riemannian volume of a smooth Riemannian manifold. 

\begin{definition}[Volume measure]

The volume measure $v_g$ of a compact stratified space endowed with an iterated edge metric $g$, is the Riemannian measure on $X^{reg}$ induced by the restriction of $g$ to this set. It is denoted as $v_g$ while the volume of a measurable set $A$ is denoted by $\vol_g(A)$.

\end{definition}

Note that the singular set has measure zero: $\vol_g(\Sigma)=0$. 

We start by observing a local property for the volume measure. Consider a point $x \in X$ and the radius $\eps_0>0$ defined as above, so that the geodesic ball $B(x,\eps_0)$ is contained in an open set homeomorphic to a truncated cone over the tangent sphere $S_x$. Denote by $\vol_C$ the volume measure associated to the cone metric $g_C$. Thanks to $\ref{g_C}$, for any regular point $y \in C_{(0,\kappa\eps_0)}(S_x^{\tiny{reg}})$ and for any vector $v$ such that $g_C(v,v)=1$ we obtain: 
$$(1-\Lambda \eps_0^{\alpha}) \leq \psi_x^*g(v,v) \leq (1+\Lambda \eps_0^{\alpha}).$$
By choosing an orthonormal basis for the cone metric $g_C$ which is orthogonal for $\psi^*_xg$, the previous implies the following inequality for the volume forms: 
\begin{equation*}
(1-\Lambda \eps_0^{\alpha})^{\frac{n}{2}}dv_C\leq dv_{\psi_x^*g} \leq (1+\Lambda \eps_0^{\alpha})^{\frac{n}{2}} dv_C. 
\end{equation*}
As a consequence, and thanks to the fact that the singular sets have null measure, for any measurable set $\mathcal{U}$ in $B(x,\eps_0)$ we have: 
\begin{equation}
\label{vol}
(1-\Lambda \eps_0^{\alpha})^{\frac{n}{2}}\vol_C(\psi_x^{-1}(\mathcal{U}))\leq \vol_g(\mathcal{U}) \leq (1+\Lambda \eps_0^{\alpha})^{\frac{n}{2}}\vol_C(\psi_x^{-1}(\mathcal{U})).
\end{equation}

The volume measure on a geodesic ball is close to the volume measure of a cone metric. This local property allows us to deduce that the volume measure of a compact stratified space is finite, $n$-Ahlfors regular and doubling.

\begin{lemma}
The volume measure $v_g$ of a compact stratified space is finite
\end{lemma}
\begin{proof}
The proof is by induction on the dimension; it is clearly true for $n=1$. Assume that any stratified space of dimension $k \leq (n-1)$ is finite. We can cover $X^n$ by finitely many geodesic balls $B(x_i,\eps_0)$ for a uniform $\eps_0$ such that $\ref{vol}$ holds on  $B(x_i,\eps_0)$. In particular, the volume of $B(x_i,\eps_0)$ with respect to $g$ is smaller than the volume of the truncated cone on the tangent sphere $S_{x_i}$ with respect to the cone metric $g_C$. If $x_i$ belongs to the regular set, $S_{x_i}$ is a sphere of dimension $(n-1)$ and $g_C$ is the round metric on an $n$-dimensional Euclidean ball. As a consequence, $\vol_g(B(x_i,\eps_0))$ is finite. If $x_i$ is a singular point, then $S_{x_i}$ is a stratified space of dimension $(n-1)$, which has finite volume by the induction hypothesis. Therefore the truncated cone over $S_{x_i}$ has finite volume with respect to $g_C$, and again $\vol_g(B(x_i,\eps_0))$ is finite. We can then cover $X^n$ by a finite number of balls of finite volume, thus $\vol_g(X)$ is finite. 
\end{proof}

%

\begin{proposition}\label{doubling}

The measure $v_g$ is $n$-Ahlfors regular: there exists a positive constant $C$ such that for any $x \in X$ and for any $0 <r < \mbox{diam}(X)/2$, the measure of the ball $B(x,r)$ is bounded as follows:
$$ C^{-1}r^n \leq \vol_g(B(x,r))\leq C r^n.$$
As a consequence, the measure $v_g$ is doubling: there exists a constant $C_1$ such that for any $x \in X$ and for any $ 0 < r < \mbox{diam}(X)/2$ 
$$\vol_g(B(x,2r)) \leq C_1 \vol_g (B(x,r)).$$
\end{proposition}

\begin{proof}
The second property is an immediate consequence of the Ahlfors regularity. Using the compactness of $X$, it suffices to prove, for all $x \in X$, the bounds
\begin{equation}\label{LocalAhlfors}
 C(x)^{-1}r^n \leq \vol_g(B(x,r))\leq C(x) r^n
 \end{equation}
for all $0<r<R(x)$ where $R(x), C(x)>0$ may depend on $x$.
Indeed, the compactness of $X$ and the property $\vol_g(X)<+ \infty $ allow us to remove the dependance in $x$ from the constant $C(x)$ and to replace $R(x)$ by $\mbox{diam}(X)$.

In order to prove \ref{LocalAhlfors} we only have to consider the case of a singular point $x \in \Sigma$. Fix $\eps_0$ as defined above and consider a geodesic ball $B(x,\eps)$, for some $\eps < \eps_0$, on which the three estimates $\ref{g_C}$, $\ref{d_C}$ and $\ref{vol}$ hold. Denote by $\mathcal{U}$ the image of $B(x,\eps)$ in the truncated cone $C_{[0,\kappa \eps)}(S_x)$ by the homeomorphism $\psi_x$. Thanks to the estimate on the distance $\ref{d_C}$, $\mathcal{U}$ must be contained in a geodesic ball with respect to the metric $g_C$ centered at the tip of the cone, whose radius is not far from $\eps$. More precisely, for $\delta= \Lambda \eps_0^{\alpha}$, $\mathcal{U}$ satisfies:
$$C_{[0,(1-\delta)\eps)}(S_x) \subset \mathcal{U} \subset C_{[0,(1+\delta)\eps)}(S_x).$$
Indeed, geodesic balls centered at the tip of the cone are truncated cone as well. 
Thanks to the estimate on the volume measure $\ref{vol}$, we obtain for any $0 < \eps \leq \eps_0$: 
$$\vol_{h_x}(S_x)(1-\delta)^n\eps^n \leq \vol_g(B(x,\eps)) \leq \vol_{h_x}(S_x)(1+\delta)^n\eps^{n}.$$
This last inequality allows us to conclude, since the volume of $S_x$ is finite.  
\end{proof}

\subsection{Analysis on stratified spaces}

We define the Sobolev space $W^{1,2}(X)$ on a compact stratified space $(X,d_g,v_g)$ as the completion of Lipschitz functions on $X$ with respect to the usual norm of $W^{1,2}$. More precisely, for a Lipschitz function $u$, the gradient $\nabla u$ is defined almost everywhere on $X$, and therefore its norm in $W^{1,2}(X)$ is given by: 
$$||u||^2_{1,2}=\int_X (u^2+|\nabla u|^2)dv_g.$$
It is possible to show that Lipschitz functions with compact support on the regular set $\mbox{Lip}_0(X^{\tiny{reg}})$, as well as $C^{\infty}_0(X^{\tiny{reg}})$ smooth functions with compact support on $X^{\tiny{reg}}$, are dense in $W^{1,2}(X)$ (see Chapter 1 of \cite{TM} for a standard proof). The Sobolev space $W^{1,2}(X)$ is clearly an Hilbert space. Moreover the usual Sobolev embeddings holding for compact smooth manifolds, also hold in the setting of compact stratified spaces, as proven in \cite{ACM12}.
 
\noindent We define the Dirichlet energy as
$$ \mathcal{E}(u)=\int_X |\nabla u|^2 dv_g, \mbox{  for } u \in C^{\infty}_0(X^{\tiny{reg}}).$$
Thanks to the density of $C^{\infty}_0(X^{\tiny{reg}})$, we can then extend $\mathcal{E}$ to the whole $W^{1,2}(X)$. The Laplacian $\Delta_g$ associated to $g$ is then the positive self-adjoint operator obtained as the Friedrichs extension of the operator generating the quadratic form $\mathcal{E}$. The integration by parts formula holds 
$$ \int_X v \Delta_g u dv_g= \int_X (\nabla u, \nabla v)_g dv_g.$$ 
It is proven in \cite{ACM12} that the spectrum of the Laplacian is discrete and nondecreasing: 
$$0= \lambda_0< \lambda_1 \leq \lambda_2 \leq \ldots \leq \lambda_n \rightarrow +\infty.$$
Moreover, in \cite{ACM14} the authors showed a local $(2,2)$-Poincaré inequality: there exists constants $a>1$, $C>0$ and $\rho_0 >0$ such that for any $x \in X$ and for any $\rho < \rho_0$, if $u \in W^{1,2}(B(x,a\rho))$ then
$$\int_{B(x,\rho)}|u-u_{B(x,\rho)}|^2 dv_g \leq C \int_{B(x,a\rho)}|\nabla u|^2dv_g,$$
where $\displaystyle u_{B(x,\rho)}=\vol_g(B(x,\rho))^{-1}\int_{B(x,\rho)}u \,dv_g$ is the average of $u$ over $B(x,\rho)$.


In the following we will need the so-called ``Sobolev-to-Lipschitz property'' which, roughly speaking, means that a Sobolev function with bounded gradient admits a Lipschitz representative (see Definition 2.3). This holds on compact stratified spaces thanks to Lemma \ref{PropLipCurve}. The precise statement is 

\begin{lemma} 
\label{LipRep}
Let $u \in W^{1,2}(X)$ a function with bounded gradient, $\nabla u \in L^2(X) \cap L^{\infty}(X)$. Then $u$ has a Lipschitz representative. 
\end{lemma}

\begin{proof}
When we restrict $u$ to $X^{\tiny{reg}}$, $u$ has a bounded gradient defined almost everywhere. Consider two points $x\neq y$ in $X^{\tiny{reg}}$. Thanks to Lemma \ref{PropLipCurve}, for any $\eps >0$ there exists an admissible curve $\gamma_{\eps}: [0,1] \rightarrow X^{\tiny{reg}}$ which connects $x$ and $y$ such that

 $$ L_g (\gamma_{\eps}) \leq (1+\eps)d_g(x,y).$$
 
 Therefore, since $\gamma_{\eps}([0,1])$ is compact, it can be covered by finitely many geodesically convex balls diffeomorphic to open Euclidean subsets. On such a ball, $u$ is locally Lipschitz with Lipchitz constant at most $\| \nabla u\|_{\infty}$. We infer from the fundamental theorem of calculus applied to $u \circ \gamma $: 
$$|u(x)-u(y)| \leq ||\nabla u ||_{\infty}(1+\eps) d_g(x,y).$$

By letting $\eps$ go to zero, we obtain that $u$ is a Lipschitz continuous function on $X^{\tiny{reg}}$ with Lipschitz constant smaller than or equal to $K=||\nabla u||_{\infty}$. This implies $u$ is uniformly continuous on $X^{\tiny{reg}}$; since the regular set is dense in $X$, $u$ admits a unique Lipschitz continuous extension $\bar{u}$ defined on the whole $X$, with $u = \bar{u}$ almost everywhere and the same Lipschitz constant.

\end{proof}

\begin{remark}
\label{twodistances}
We can consider another possible metric structure on the stratified space, by defining the distance associated to the Dirichlet energy as follows: 
$$d_{\mathcal{E}}(x,y)=\sup\{|f(x)-f(y)|; f \in W^{1,2}(X), ||\nabla f||_{\infty} \leq 1 \}.$$
Thanks to the previous result, we know that Sobolev functions with bounded gradient are Lipschitz functions, and $||\nabla f ||_{\infty}\leq 1$ implies they have Lipschitz constant at most one. As a consequence,
for any $f$ as in the definition and for any $x,y \in X$ we have $|f(x)-f(y)|\leq d_g(x,y)$, which implies $d_{\mathcal{E}}(x,y) \leq d_g(x,y)$. Moreover, note that for any $x$ fixed, $f(y)=d_{g}(x,y)$ is clearly a Lipschitz function of Lipschitz constant one, belonging to $W^{1,2}(X)$. Therefore for any $x,y \in X$ we have $d_{\mathcal{E}}(x,y) \geq d_g(x,y)$. Then the two distances coincides.
\end{remark}

\subsubsection{Ultracontractivity} In the proof of Theorem \ref{MainThm} we will  use some properties of the heat semi-group that we recall here. Let $X$ be a compact stratified space, and let $g$ be an iterated edge metric. By definition, $\mathcal{E}$ is also a \textit{strongly regular Dirichlet form}.
Let $(P_t)_{t>0}$ denote the associated heat semi-group. 
$\mathcal{E}$ is strongly local, admits
a local $(2,2)$-Poincaré inequality, and the measure $v_g$ satisfies a doubling property. Moreover, closed balls w.r.t. $d_{\mathcal{E}}$ are compact.
Therefore, we can apply results from \cite{sturmdirichlet1, sturmdirichlet2, sturmdirichlet3} and, in particular, the following lemma holds.
\begin{lemma}
\label{ultracontractive}
$P_t$ is $L^1\rightarrow L^{\infty}$-ultracontractive. More precisely, for every $t\in (0,\infty)$ there exists a constant $C(t)>0$ such that
$
\left\|P_t\right\|_{L^1(\vol_g)\rightarrow L^{\infty}(\vol_g)}\leq C(t).
$
\end{lemma}
\noindent
\textit{Proof of the Lemma:} The assumptions imply a uniform bound $\tilde{C}(t)>0$ on the associated heat kernel $(x,y)\mapsto p_t(x,y)$ by \cite[Theorem 0.2]{sturmdirichlet2} or by  \cite[Corollary 4.2]{sturmdirichlet3}. Then, we can first deduce 
$L^1\rightarrow L^p$-ultra-contractivity for some $p>1$ (for instance, compare with \cite[Chapter 14.1]{grigoryanheat}), 
and this property implies our claim (for instance, see again \cite[Chapter 14.1]{grigoryanheat} or \cite[Theorem 6.4]{agsriemannian}).\qed

%% file: RCD.tex
\section{A minimal introduction to analysis on Metric Measure Spaces}

\subsection{Calculus on metric measure spaces}
In this part, we provide a minimal introduction on the analytical tools used in the theory of $RCD$-spaces.
We follow closely the approach of Ambrosio, Gigli and Savaré \cite{agslipschitz, agsheat, agsriemannian, agsbakryemery}.

Throughout this section $(X,d)$ is a complete and separable metric space, and let $\m$ be a locally finite Borel measure. The triple $(X,d,\m)$ is then called a \textit{metric measure space}.
Moreover, we assume from now on that the so-called \textit{exponential volume growth condition} holds
\begin{align}
\label{volumeGrowth}
\exists x_0\in X, \ \exists C>0: \ \int_Xe^{-C d(x,x_0)^2}d\m<\infty.
\end{align}
\subsubsection{Cheeger energy and Sobolev space}
As for a stratified space we denote by $\Lip(X)$ the set of Lipschitz functions on $(X,d)$, and for $f\in \Lip(X)$ we define the the \textit{local slope} or 
the \textit{local Lipschitz constant} $\lip(f)$ as 
\begin{align*}x\mapsto
\lip(f)(x)=\limsup_{y\rightarrow x}\frac{|f(x)-f(y)|}{d(x,y)}.
\end{align*}
$\Lip_1(X)$ denotes the set of Lipschitz functions with local slope bounded from above by $1$.
Then, the \textit{Cheeger energy} of $(X,d,\m)$ is defined via
\begin{align*}
\Ch:L^2(\m)\rightarrow[0,\infty], \ \ \Ch(f)=\frac{1}{2}\liminf_{\Lip(X)\ni f_n\overset{L^2}{\rightarrow} f} \int_X\lip(f)^2d\m.
\end{align*}
The \textit{Sobolev space} $D(\Ch)$ of $(X,d,\m)$ is given by
$$D(\Ch)=\left\{f\in L^2(\m): \Ch(f)<\infty\right\}$$
and equipped with the norm $\left\|f\right\|_{D(\Ch)}^2=\left\|f\right\|_{2}^2+2\Ch(f)$ where in this context $\left\|f\right\|_{2}$ denotes the Lebesgue $L^2$-norm w.r.t. $\m$. Note that $D(\Ch)$ with $\left\|\cdot\right\|_{D(\Ch)}$ is not a Hilbert space in general.
For instance, a Banach space $V$ that is not a Hilbert space, or more general any Finsler manifold that is not Riemannian will generate a space of Sobolev functions that is not a Hilbert space as well.

\subsubsection{Minimal relaxed and weak upper gradient}
A function $g\in L^2(\m)$ is called a \textit{relaxed gradient} of $f\in L^2(\m)$ if there exists a sequence of Lipschitz functions $(f_n)_{n\in\mathbb{N}}$ such
that $(f_n)$ converges to $f$ in $L^2(\m)$, and there exists $\tilde{g}\in L^2(\m)$ such that $\lip(f_n)$ weakly converges to $\tilde{g}$ in $L^2(\m)$ and $g\geq \tilde{g}$.
We call $g$ the \textit{minimal relaxed gradient} of $f$ if it is minimal w.r.t. the norm amongst all relaxed gradients. We write $|\nabla f|_*$ for the minimal relaxed gradient of $f$. 
Any $f\in D(\Ch)$ admits a minimal relaxed gradient, and $\Ch(f)=\frac{1}{2}\int|\nabla f|_*^2 d\m$.

An alternative approach is to introduce so-called \textit{weak upper gradients} for an $L^2$-function $f$. Then one can define uniquely the so-called \textit{minimal weak upper gradient} $|\nabla f|_w$. We will omit any details about the definition. However, let us mention this notion is inspired by Cheeger's work \cite{Cheeger_dif} where he defined the notion of (minimal) generalised upper gradient $|\nabla f|_w$ . The author proved \cite[Theorem 5.1]{Cheeger_dif} that on a complete length space $(X,d,m)$ that is doubling and which supports a local $(2,2)$-Poincaré inequality, 
\begin{equation}\label{1}
 |\nabla f|_w =\lip (f)
\end{equation}
holds $m$-a.e. space where $f$ is any locally Lipschitz function on $X$.

We have seen in Section \ref{sec-pre-stra} that a stratified space meets these assumptions, thus Cheeger's result applies in our setting. Moreover, by combining  \cite[Theorem 6.2]{agsheat} with earlier work by Shanmugalingam \cite{Shan} (see \cite{agsheat} for more details), it can be proved that
\begin{equation}\label{2}
|\nabla f|_{*}= | \nabla f|_w \;\;\;m-a.e.
\end{equation} 

This result holds  on any complete separable metric measure space satisfying some mild assumption on $m$ (see \cite{agsheat}). This set of spaces comprises  compact stratified spaces and thus the combination of (\ref{1}) and (\ref{2}) gives a proof of 

\begin{equation}\label{3}
|\nabla f|_{*}= \lip(f) \;\;\;m-a.e.
\end{equation} 
for any locally Lipschitz function $f$ on $X$.

 Because none of these results is elementary, we provide another short proof which applies in our particular setting in the appendix, see Proposition \ref{gradient}.

\begin{remark}
For an $L^2$-integrable Lipschitz function the local Lipschitz constant can be strictly bigger than the relaxed gradient. For instance, consider $\mathbb{R}^n$ equipped with a measure 
that is the sum of finitely many Dirac measures. Then, the Cheeger energy for any Lipschitz function on $\mathbb{R}^n$ is $0$.
\end{remark}

\subsubsection{Infinitesimally Hilbertian metric measure spaces}
If $D(\Ch)$ is in fact a Hilbert space, we say that $(X,d,\m)$ is \textit{infinitesimally Hilbertian}. In this case we can define a \textit{pointwise inner product} between minimal relaxed gradients by
\begin{align*}
(f,g)\in D(\Ch)^2\mapsto \langle \nabla f,\nabla g\rangle := \frac{1}{4}|\nabla (f+g)|_*^2- \frac{1}{4}|\nabla(f-g)|_*^2.
\end{align*}
We say $f\in D(\Ch)$ is in the \textit{domain of the Laplace operator} if there exists $g\in L^2(\m)$ such that for every $h\in D(\Ch)$ we have 
\begin{align*}
\int\langle \nabla f,\nabla h\rangle d\m=-\int hgd\m.
\end{align*}
We say $f\in D(\Delta)$.
If $f\in D(\Delta)$, then $g\in L^2(\m)$ as above is uniquely determined, and we write $g=\Delta f$. Note that the definition of $D(\Delta)$ intrinsically sets \textit{Neumann boundary condition}.
$D(\Delta)$ is equipped with the so-called \textit{operator norm} $\left\|f\right\|_{D(\Delta)}^2=\left\|f\right\|_{2}^2+\left\|\Delta f\right\|^2_{2}$. We also define 
\begin{align*}
D_{D(\Ch)}(\Delta)=\left\{f\in D(\Delta):\Delta f\in D(\Ch)(X)\right\}
\end{align*}
and similar $D_{L^{\infty}(\m)}(\Delta)$.

\begin{remark}
We emphasize that in the context of $RCD$ spaces the sign convention for $\Delta$ differs from the one that we chose for the Laplace operator $\Delta_g$ on stratified spaces.
\end{remark}

\subsubsection{Bakry-Émery curvature-dimension condition} Another way to define curvature-dimension conditions was introduced by D.~Bakry (see for example \cite{BakryStFlour}) using the so-called $\Gamma$-\emph{calculus}, based on the Bochner inequality on manifolds with a lower Ricci bound. In the following we mainly refer to \cite{agsbakryemery}.
Let $(X,d,\m)$ be a metric measure space that is infinitesimally Hilbertian. 
For $f\in D_{D(\Ch)}(\Delta)$ and $\phi\in D_{L^{\infty}}(\Delta)\cap L^{\infty}(\m)$ we define the \textit{carr\'e du champ operator} as
\begin{align*}
\Gamma_2(f;\phi)=\int \frac{1}{2}|\nabla f|_*^2\Delta \phi d\m - \int\langle\nabla f,\nabla \Delta f\rangle \phi d\m.
\end{align*}
\begin{definition}[Bakry-Émery condition]\label{def:BE}
We say that $(X,d,\m)$ satisfies the \textit{Bakry-Émery condition} $BE(\ke,N)$ for $\ke\in \mathbb{R}$ and $N\in (0,\infty]$ if it satisfies the weak Bochner inequality
\begin{align*}
\Gamma_2(f;\phi)\geq \frac{1}{N}\int (\Delta f)^2 \phi d\m + \ke\int |\nabla f|_*^2 \phi d\m.
\end{align*}
for any $f\in D_{D(\Ch)}(\Delta)$ and any test function $\phi\in D_{L^{\infty}}(\Delta)\cap L^{\infty}(\m), \, \phi\ge 0$.
\end{definition}

Under some mild assumptions, the Bakry-Émery condition is equivalent to several notions of metric measure space (mms) with "Ricci curvature bounded below". This subject has received a lot of attention over the last fifteen years with the introduction of several notions of curvature-dimension conditions on mms. Among them, we distinguish the $RDC(K,N)$ spaces and $RDC^*(K,N)$ spaces , \cite{giglistructure, agsriemannian,EKS}. These two notions are actually equivalent when the measure of the space is finite, as recently proved by Cavalletti and Milman \cite{CavallettiMilman}.  In order to state the theorem, we first define:

\begin{definition}\label{def:sobtolip}
We say a metric measure space satisfies the \textit{Sobolev-to-Lipschitz} property \cite{giglistructure} if
\begin{align*}
\left\{f\in D(\Ch):|\nabla f|_*\leq 1\ \m\mbox{-a.e.}\right\}\subset\Lip_1(X).
\end{align*}
\end{definition}

\begin{theorem}[\cite{EKS, AMS15}]\label{th:eks}
Let $(X,d,\m)$ be a metric measure space {{satisfying \ref{volumeGrowth}}}. Then, the condition $RCD(\ke,N)$ for $\ke\in \mathbb{R}$ and $N>1$ holds if and only if $(X,d,\m)$ is infinitesimally Hilbertian, it satisfies the 
Sobolev-to-Lipschitz property and it satisfies the Bakry-Émery condition $BE(\ke,N)$.
\end{theorem}


\subsubsection{Cheeger versus Dirichlet energy}
Our strategy to prove Theorem \ref{MainThm} consists in applying the previous theorem to a compact stratified space $(X,d_g,v_g)$ endowed with the structure of a metric measure space introduced in the first section. The assumption \ref{volumeGrowth} is clearly satisfied since the volume of a compact stratified space is finite. 
In order to be able to apply Theorem \ref{th:eks}, we need the Sobolev space as defined in the previous section to agree with the domain of the Cheeger energy, and the different notions of gradients to be equivalent for Sobolev functions.

We have proven that for any locally Lipschitz function $u$, $|\nabla u|_* = \lip (u) $ \ref{3}. 
By density of Lipschitz functions in the domains of both the Dirichlet energy and the Cheeger energy, such domains coincide:
$$W^{1,2}(X)=D(\Ch).$$
As a consequence, since $W^{1,2}(X)$ is a Hilbert space, the same is true for $D(\Ch)$: any compact stratified space is infinitesimally Hilbertian. 

In particular, for any $u\in W^{1,2}(X)$ we have 
\begin{equation}
\label{ChD}
2\Ch(u)=\int_X |\nabla u |_*^2 dv_g = \int_X |\nabla u|_g^2 dv_g = \mathcal{E}(u),
\end{equation}
where $|\cdot|_g$ is the usual norm with respect to the iterated edge metric $g$.
The density of the Lipschitz functions together with $\ref{ChD}$ also guarantees that for any Sobolev function $u$ we have $|\nabla u|_g = |\nabla u|_*$ almost everywhere. Therefore, the fact that a Sobolev function has a Lipschitz representative, proven in Lemma \ref{LipRep}, also proves that the Sobolev-to-Lipschitz property, as stated in \ref{def:sobtolip}, holds on compact stratified spaces.

Clearly, the Laplace operator $\Delta$ in the sense of a metric measure space is the Laplace operator with Neumann boundary condition in the sense of stratified spaces up to a minus sign: $\Delta=-\Delta_g$.

Therefore, we can apply Theorem \ref{th:eks}: more precisely, establishing the Bakry-Émery curvature dimension condition $BE(K,n)$ for the Dirichlet energy $\mathcal{E}$ and its Laplace operator will imply the Riemannian curvature-dimension condition $RCD(K,n)$ for a compact stratified space $(X,d_g,v_g)$ with a singular lower Ricci curvature bound.

%% file: examp.tex
In this section we recall and make more precise the definition of a singular lower Ricci bound presented in the introduction; in the following we also illustrate some examples of stratified spaces with a singular Ricci lower bound.

\begin{definition}[Singular Ricci lower bound]
\label{RicciBD}
Let $X$ be a compact stratified space of dimension $n$ endowed with an iterated edge metric $g$. Let $K \in \R$. We say  that $g$ has {\it singular} Ricci curvature bounded from below by $K$ if
\begin{itemize}
\smallskip
 \item[(i)] $\ric_g\geq K$ on the regular set $X^{reg}$,
 \medskip
 \item[(ii)] for every $x\in \Sigma^{n-2}$ we have $\alpha_x \leq 2\pi$.
\end{itemize}
\end{definition}
Lemma 1.1 in \cite{M14} proves that if the Ricci tensor of $g$ is bounded below on the regular set of $X$, then the regular set of each tangent cone $C(S_x)$ carries a metric with non-negative Ricci tensor. As a consequence, for each link $(Z_j, k_j)$ we have $\ric_{k_j}\geq (j'-1)$ on $Z_j^{\mbox{\tiny{reg}}}$, where $j'= n-j-1$ is the dimension of the link. Observe that when the codimension of the strata is strictly larger than 2, the previous implies that the link carries a metric with strictly positive Ricci tensor. Moreover, when we consider the regular set of $C(Z_j)$, that is $C(Z_j^{\mbox{\tiny{reg}}})\setminus \{ 0\}$, this is an open manifold with non-negative Ricci tensor, as observed by J.~Cheeger and M.~Taylor (see \cite{C1}, \cite{CT}). 

As for the stratum of codimension 2, if we only assume that $g$ has Ricci tensor bounded below, we only get that the Ricci tensor on $\s^1$ is non-negative, and we cannot deduce any positivity for the curvature of the two-dimensional cone $C(\s^1)$. In order to have a bound by below for the curvature, in the sense of Alexandrov, or with respect to the curvature-dimension condition $\mbox{CD}(0,2)$ (see \cite{bastco}), we need to assume that the diameter of $(\s^1,a_x^2d\theta^2)$ is less than or equal to $\pi$; equivalently the radius $a_x$ need to be smaller than or equal to one, and the angle $\alpha_x$ is less than or equal to $2\pi$. Since the angle may depend on the point $x \in \Sigma^{n-2}$, we need to assume condition (ii) for all points of the stratum of codimension two.
 
We present some examples of stratified spaces carrying a metric with a singular lower Ricci bound. Thanks to Theorem \ref{MainThm}, all of these examples are $RCD(K,N)$ spaces. Most of them are previously unknown examples and some of them recover the known examples of orbifolds and spherical suspension over smooth manifolds with a Ricci lower bound. 

We point out that, except for the case of orbifolds, all the examples in the following can be constructed in order to have a Ricci lower bound on the regular set and angles $\alpha_x$ along the stratum of codimension two \emph{larger} than $2\pi$. For such examples, having a lower Ricci bound on the regular set does not suffice to satisfy the $RCD(K,N)$ condition. 
\\

\noindent \textbf{Manifolds with isolated conical singularities.} A compact surface with isolated conical singularities of angle less than $2\pi$ and non-negative sectional curvature is known to be an Alexandrov space. It is clearly a simple example of a stratified space with Ricci tensor bounded below as in Definition \ref{RicciBD}. More generally, if we add isolated conical singularities of angle less than $2\pi$ to a compact smooth manifold with Ricci tensor bounded below, we obtain a stratified space satisfying \ref{RicciBD}. \\

\noindent \textbf{A construction of a singular stratum.} Consider a sphere $\s^3$ with round metric $g_0$, and a closed circle $c$ in $\s^3$. By using Fermi coordinates in a tubular neighbourhood $\mathcal{U}^{\eps}$ of $c$ of size $\eps$, it is possible to write the metric in $\mathcal{U}^{\eps}$ as a perturbation of the following product metric: 
$$dr^2+r^2d\varphi^2+a^2d\theta^2,$$
where $a$ is the radius of the circle. More precisely, there exists a positive constant $\Lambda$ such that: 
$$|g_0 - (dr^2+r^2d\varphi^2+a^2d\theta^2)| \leq \Lambda r^{\gamma},$$
where $\gamma=1$ if $c$ is not totally geodesic, $\gamma=2$ otherwise. We refer to the appendix for the details. Now, we can choose $\alpha \in [0,2\pi]$ and modify the metric in $\mathcal{U}^{\eps}$ so that the new metric  does not change outside of $\mathcal{U}^{\eps}$ and it is a perturbation of the singular metric: 
$$ g_{\alpha}= dr^2+\left(\frac{\alpha}{2\pi} \right)^2r^2d\varphi^2+a^2d\theta^2;$$
This makes the circle $c$ a singular stratum of codimension two and angle equal to $\alpha$. Moreover, this construction leaves the Ricci tensor of $g_0$ bounded below away from $c$; therefore we constructed a simple stratified space with Ricci tensor bounded below as in Definition \ref{RicciBD}. The same construction can be done along a codimension two submanifold in any compact smooth manifold with Ricci tensor bounded below. \\

\noindent \textbf{Singular space associated to a static triple.} A static triple is a triple $(M^n,g,V)$ where $(M^n,g)$ is a complete manifold with boundary $\partial M$ and $V$ a static potential, that is a non trivial solution $V \in C^{\infty}(M)$ to the equation
\begin{equation*}
\nabla^2V-(\Delta_gV)g-V Ric_g=0.
\end{equation*}
Static triples have been studied in general relativity and in differential geometry, in the context of prescribing scalar curvature; in the following we mainly refer to \cite{Ambrozio}. A static triple always has constant scalar curvature, which can be renormalized to be equal to $\varepsilon n(n-1)$ with $\varepsilon \in \{+1,0,-1\}$, the boundary $\partial M$ is totally geodesic and $|\nabla V|$ is constant on each connected component of $\partial M$ (Lemma 3 in \cite{Ambrozio}). 

Starting from a static triple, it is possible to construct an associated singular space, which turns out to be an Einstein stratified space. This construction has long been known in the setting of general relativity (for example \cite{GibbonsHawking}); we refer here to Section 6 in \cite{Ambrozio} for the precise details and only recall the main features of the singular space. For any static triple $(M^n,g,V)$ there exists a stratified space $(N^{n+1},h)$ with one singular stratum $ \Sigma$ of codimension 2 which can be identified with $\partial M$, thus can be disconnected. The regular set $X^{\tiny{reg}}$ of $(N^{n+1},h)$ is isometric to the product $\s^1\times (M \setminus \partial M)$ and the metric $h$ on $X^{\tiny{reg}}$ is Einstein with $Ric_h=\varepsilon n h$. The angles along each connected component of $\Sigma$ are determined by the value of $|\nabla V|$. Observe that the stratified space $(N^{n+1},h)$ is compact if and only if $(M^n,g)$ is compact; in this case, the static potential $V$ can be renormalized so that $|\nabla V|\leq 1$ on each connected component of $\partial M$: this implies that the angles along the stratum are smaller than $2\pi$. As a consequence, in the compact case $(N^{n+1},h)$ is a stratified space with Ricci tensor bounded below in the sense of Definition \ref{RicciBD}. \\

\noindent \textbf{Kähler-Einstein manifolds with a divisor.} In \cite{JMR}, T.~Jeffres, R.~Mazzeo and Y.~Rubinstein considered compact Khäler manifolds with a smooth divisor $D$, carrying a metric with angle $\alpha \in (0,2\pi]$  along $D$. The divisor is a singularity of codimension two, and such manifolds belong to the setting of stratified spaces. 

The authors proved the existence of a Kähler-Einstein metric $g$ on $M\setminus D$, whose asymptotic along $D$ has angle $\beta \in (0,2\pi)$. Therefore, $M$ endowed with the Kähler-Einstein metric $g$ is a stratified space with Ricci tensor bounded below as in \ref{RicciBD}.

Note that the existence of a Kähler-Einstein metric with edge singularity has been an important step towards the proof that any smooth K-stable Fano manifold carries a Kähler-Einstein metric (see \cite{CDonS1,CDonS2,CDonS3} and \cite{Tian1}).\\

\noindent \textbf{Orbifolds} Any compact Riemannian $n$-orbifold without boundary is a stratified space (see \cite{TM}). If the regular set of the orbifold has Ricci tensor bounded below, then the orbifold satisfies Definition \ref{RicciBD}. In fact, all the links are quotients of a sphere $\s^k$, for $1\leq k \leq (n-1)$ by a finite group of isometries; even in the case of the stratum of codimension two, and $k=1$, the link is a circle of diameter less than or equal to $\pi$, without any further assumption. Theorem A applied to compact orbifolds without boundary partially recovers Theorem 7.10 in \cite{GalazKellMondinoSosa}\\


\noindent \textbf{Spherical suspension.} Consider a circle $(\s^1,a^2d\theta^2)$ of radius $a$ smaller than one, and the following spherical suspension:
\begin{align*}
S^n_{\alpha}&=[0,\frac{\pi}{2}]\times \s^{n-2}\times \s^1\\
g_{\alpha}&=d\varphi^2+\cos^2(\varphi)g_{S^{n-2}}+\left(\frac{\alpha}{2\pi} \right)^2\sin^2\varphi d\theta^2,
\end{align*}
where $g_{\s^{n-2}}$ is the round metric of the unit sphere $\s^{n-2}$ and $\alpha=2\pi a$. Then $(S^n_{\alpha},g_{\alpha})$ is a compact stratified space with singular set of codimension 2 and angle $\alpha \leq 2\pi$. Moreover, it is easy to check that $g_{\alpha}$ is an Einstein metric with $Ric_{g_{\alpha}}=(n-1)$. Therefore $(S^n_{\alpha},g_{\alpha})$ is a compact stratified space with Ricci tensor bounded below. 

More generally, if we consider a compact smooth manifold $(M^n,g)$ of dimension $n\geq 2$ and with $Ric_g \geq (n-1)$, the spherical suspension $([0,\pi]\times M, dt^2+\sin^2(t)g)$ is a compact stratified space satisfying Definition \ref{RicciBD}. Therefore, Theorem A agrees with previous results of \cite{bastco}.  Moreover, the spherical suspension of a compact stratified space satisfying Definition \ref{RicciBD} is also a compact stratified space with a singular Ricci lower bound. 

%% file: geometric.tex
\section{Geometric consequences of the curvature-dimension condition}

Thanks to Theorem \ref{MainThm} we know that a compact stratified space $(X^n,g)$ with singular Ricci curvature bounded below by $K$ is a $RCD(K,n)$ metric measure space. This allows us to apply to stratified spaces several geometric results that are known in the setting of $RCD(K,n)$ metric measure spaces and of smooth Riemannian manifolds, but are new in the case of stratified spaces. Moreover, the $RCD(K,N)$ condition can be used to obtain a characterization of compact stratified spaces with curvature bounded below in the sense of Alexandrov. 

\subsection{Essential non-branching} The first consequence of the $RCD(K,n)$ condition is that $(X^n,g)$ is \emph{essentially non-branching}. A metric measure space $(X,d,\mathfrak{m}$) is said to be essentially non-branching if for any two measures $\mu_0,\mu_1$ in the Wasserstein space $\mathcal{P}^2(X,\mathfrak{m})$, absolutely continuous with respect to $\mathfrak{m}$, any optimal plan $\pi$ between $\mu_0$ and $\mu_1$ is concentrated on a set of non-branching geodesics. The fact that a $RCD(K,\infty)$ space is essentially non-branching has been proven in \cite{RajalaSturm}. Thus a stratified space with singular Ricci curvature bounded below is essentially non-branching. 
We point out that essential non-branching does not exclude the existence of branching geodesics, which may occur in the setting of stratified spaces. Nevertheless, examples of branching $RCD(K,N)$ spaces are not known. 

\subsection{Bishop-Gromov} A direct consequence of the $RCD(K,N)$ condition is the Bishop-Gromov volume estimate. This has been proven for $CD(K,N)$ spaces by K.-T.~Sturm in \cite[Theorem 2.3]{stugeo2}. Let $\omega_n$ be the volume of the unit ball in the Euclidean space and define:
$$v_k(r)=n\omega_n \int_0^r \sin_k(t)^{n-1}dt,$$
that is the volume of a ball of radius $r$ in the $n$-dimensional space form of constant curvature $k$. Then the following holds:

\begin{cor}[Bishop-Gromov volume estimate] Let $(X^n,d_g,v_g)$ be a stratified with singular Ricci curvature bounded below by $K$ and $n > 1$. Then for any $0< r < R <\mbox{diam}(X)$ we have: 
$$ \frac{\mbox{\emph{vol}}_g(B(x,r))}{\mbox{\emph{vol}}_g(B(x,R))}\leq \frac{v_{K/(n-1)}(r)}{v_{K/(n-1)}(R)}.$$
\end{cor}

Observe that if $x$ is a point in the regular set, the Bishop-Gromov volume estimate holds for a sufficiently small radius due to the fact that the Ricci tensor is bounded below on $X^{reg}$. The result is new when we consider a point $x$ in a singular stratum or when the radii are large. 

\subsection{Laplacian comparisons} We refer to N.~Gigli's proof of Laplace comparisons for $CD(K,N)$ spaces, in Theorem 5.14 and Corollary 5.15 of \cite{giglistructure}. Note that both of these results need the metric measure space $(X,d,\mathfrak{m})$ to be compact and $q$-infinitesimally strictly convex for some real $q$. When $(X,d,\mathfrak{m})$  is infinitesimally Hilbertian, it is 2-infinitesimally strictly convex, as proven in \cite{agsriemannian}; therefore, Laplace comparisons as given in \cite{giglistructure} hold in any compact $RCD(K,N)$ space. We point out that F.~Cavalletti and A.~Mondino \cite{CM18} recently proved Laplacian comparisons for the distance function removing the assumptions of infinitesimal strict convexity and compactness of the space.

For the sake of completeness, we state the result in the setting of compact stratified spaces for the distance function. As in the case of the Bishop-Gromov volume estimate, the interest of the result is for the distance function to a point in the singular set or because the estimate also holds far away from the base point.

We define the measure valued Laplacian as follows: 

\begin{definition}
Let $E \subset X$ be an open set in $(X^n,g)$, $\mu$ a Radon measure concentrated on $E$, and $f$ a locally Lipschitz function on $E$. We say that $f$ has distributional Laplacian  bounded from above by $\mu$ on $E$ if for any nonnegative $\varphi \in \mbox{Lip}_0(E)$, the following holds: 
$$\int_X (\nabla f,\nabla \varphi)_g dv_g \leq \int_X \varphi \,d\mu.$$ 
In this case, we write $\Delta_g f \leq \mu$ on $E$.
\end{definition}

\begin{cor}
Let $(X^n,g)$ be a compact stratified space with singular Ricci curvature bounded from below by $k(n-1) \in \R$. For $x\in X$, let $d_x$ be the distance function from $x$, namely $d_x(y)=d_g(x,y)$. Then  the following inequalities hold: 
\begin{equation*}
\Delta_g d_x \leq (n-1) \frac{\sin_k'(d_x)}{\sin_k(d_x)}dv_g \, \mbox{ on } X \setminus \{x\}, \quad \Delta_g  \left(\frac{d_x^2}{2}\right) \leq n \,d_x \frac{\sin_k'(d_x)}{\sin_k(d_x)}dv_g  \mbox{ on } X.
\end{equation*}
\end{cor}

This clearly corresponds to the situation in the smooth setting. 

\subsection{Lévy-Gromov isoperimetric inequality} A classical and well known isoperimetric inequality for smooth manifolds is the Lévy-Gromov isoperimetric inequality. Let $(M^n,g)$ be a compact smooth manifold with Ricci tensor bounded below by $(n-1)$, and consider a domain $\Omega$ in $M$ with smooth boundary and volume 
$$ \vol_g(\Omega)=\beta.$$
Denote by $B_{\beta}$ a geodesic ball in the standard sphere $\s^n$ with volume equal to $\beta \vol(\s^n)$. If equality
$$ \frac{\vol_g(\Omega)}{\vol_g(M)}=\frac{\vol(B_{\beta})}{\vol(\s^n)}$$
holds, then we have:
\begin{equation}
\label{LG}
\frac{\vol_g(\partial \Omega)}{\vol_g(M)} \geq \frac{\vol(\partial B_{\beta})}{\vol(\s^n)},
\end{equation}
where we denote by $\vol$ the Riemannian volume in the round sphere $\s^n$. Moreover, the equality in \ref{LG} holds if and only if $(M^n,g)$ is isometric to the standard sphere and $\Omega$ is isometric to the geodesic ball $B_{\beta}$.

In \cite{CavallettiMondino1} the authors proved an analog result in the setting of $RCD^*(K,N)$ spaces for $K >0$ and $N \geq 2$. We state it in the setting of compact stratified spaces. We consider $(X^n,g,\mathfrak{m})$ a compact stratified space with the renormalized measure
$$\mathfrak{m}=\vol_g(X)^{-1}dv_g,$$
so that $\mathfrak{m}(X)=1$. The outer Minkowski content of an open set $E \subset X$ is used to measure the size of the boundary. It is defined by: 
$$ \mathfrak{m}^+(E)=\liminf_{\eps \rightarrow 0^+}\frac{\mathfrak{m}(E^{\eps})-\mathfrak{m}(E)}{\eps},$$
where $E^{\eps}$ is the tubular neighbourhood of size $\eps$ of $E$ with respect to the distance $d_g$. 
%
%

Then, thanks to Theorem 1.1 in \cite{CavallettiMondino1} we have:

\begin{cor}[Lévy-Gromov isoperimetric inequality]
Let $(X^n,g,\mathfrak{m})$ be a compact stratified space with singular Ricci curvature bounded below by $n-1 >0$. Then for every open set $E \subset X$ the following inequality holds:
$$ \mathfrak{m}^+(E) \geq \frac{\vol(\partial B_{\beta})}{\vol(\s^n)},$$
where $\beta=\mathfrak{m}(E)$ and $B_{\beta}$ is a geodesic ball in $\s^n$ of volume $\vol(B_{\beta})=\beta \vol(\s^n)$. 
\end{cor}

We also obtain the following rigidity result: 

\begin{cor}[Rigidity in Lévy-Gromov]
Let $(X^n,g,\mathfrak{m})$ be a compact stratified space with singular Ricci curvature bounded below by $(n-1)$. If there exists an open domain $E$ in $X$ satisfying:
$$ \mathfrak{m}^+(E) = \frac{\vol(\partial B_{\beta})}{\vol(\s^n)},$$
then there exists a compact stratified space $(Y^{n-1},h)$ with singular Ricci curvature bounded below by $(n-2)$ such that $(X^n,g)$ is isometric to the spherical suspension $([0,\pi]\times Y, dt^2+\sin^2(t)h)$.
\end{cor}

\begin{proof}
In \cite{CavallettiMondino1}, the authors prove that if equality holds in the Lévy-Gromov isoperimetric inequality, then $X$ must have diameter equal to $\pi$. Therefore, by Theorem 2.3 and 3.1 in \cite{M15}, $(X^n,g)$ must be isometric to a spherical suspension of a stratified space satisfying the analog bound on the singular Ricci curvature. Another way to prove the same, is by using Theorem 1.4 in \cite{CavallettiMondino1}, which tells us that the stratified space is isometric to a spherical suspension of a $RCD(n-2,n-1)$ space $Y$. In particular, the spherical suspension is a stratified space and tangent cones at all points are metric cones over a stratified space of dimension $(n-1)$. Now, then tangent cone at the points $\{0\}\times Y$ and $\{\pi\}\times Y$ is the metric cone over $Y$. As a consequence, $Y$ is also a stratified space.
\end{proof}

%
%

\subsection{Weyl law} Let $\{\lambda_i\}_{i \in \N}$ be the sequence of eigenvalues of the Laplacian $\Delta_g$. For any $\lambda >0$ we define: 
$$ N(\lambda)=\sharp\{\lambda_i,\mbox{such that } \lambda_i \leq \lambda\}.$$
A well-known result on smooth manifolds states that the asymptotics of $N(\lambda)$ as $\lambda$ tends to infinity is given by $\lambda^{-n/2}$ times a constant which depends on the volume and on the dimension of the manifold. An analog result has been proven in \cite[Corollary 4.8]{AHT} in the setting of $RCD^*(K,N)$ spaces, when the measure is Ahlfors $n$-regular for some $n \in \N$; the Riemannian volume is replaced by the $n$-dimensional Hausdorff measure of the space.
This result clearly applies to stratified spaces with singular Ricci curvature bounded from below: 
\begin{cor}
Let $(X^n,d_g,v_g)$ be a stratified space with singular Ricci curvature bounded from below. Then we have:
$$ \lim_{\lambda \rightarrow +\infty}\frac{N(\lambda)}{\lambda^{\frac n2}}=\frac{\omega_n}{(2\pi)^n}\vol_g(X).$$
\end{cor}
Observe that in the smooth setting the Weyl law holds without any assumption on the Ricci curvature. It is then reasonable to believe that on stratified spaces too, the hypothesis of singular Ricci curvature bounded from below could be dropped. 

\subsection{Further properties of geodesics and $\CBB$ stratified
  spaces} 
\label{sec:furth-prop-geod}

The applications of Theorem \ref{MainThm} described in this subsection
were pointed to us by V. Kapovitch. 

The $RCD(K,N)$ property can be used to gain further knowledge on the
behavior of geodesics: we will be able to show that the regular set
$X^{reg}$ of a stratified space $(X,g)$ with singular Ricci curvature
bounded below is almost everywhere convex. 

This will in turn
allow us to use a theorem of N. Li to prove the analogue of Theorem
\ref{MainThm} in the presence of lower bounds on the sectional curvature.
Besides the work of N. Li \cite{NanLi} which considers such probabilistic convexity
properties in the context of Alexandrov geometry, these have also been
investigated for Ricci limit spaces, see \cite{ColdingNaber} and the
references therein.

The proof of the almost everywhere convexity of $X^{reg}$ relies on
the measure contraction property $MCP(K,N)$, which is a consequence of
the $RCD(K,N)$ property, see \cite{cavmon}, end of section 5.

A subset $U$ of a geodesic metric measure space $(X,d,m)$ is said to be 
$m$-almost everywhere convex if for every $x\in U$ : \[m\left(\{y\in X|\text{no
  minimizing geodesic from $x$ to $y$ is included in $U$}\}\right)=0.\]

\begin{proposition}
  Let $(X,g)$ be a compact stratified space with singular Ricci curvature
bounded below by some $K\in\mathbb{R}$, then $X^{reg}$ is 
$v_g$-almost everywhere convex.
\end{proposition}
\begin{proof}
  Without loss of generality we assume that $(X,g)$ has volume 1, so
  that $v_g$ is a probability measure.

   Let $\mathcal{G}(X)$ denote the set of constant speed geodesics
  $\gamma:[0,1]\to X$ and $\Sigma$ denote the singular set of $X$.

  Since $(X,d_g)$ is $RCD(K,N)$, it is essentially non branching and
  $CD(K,N)$, and will satisfy a measure contraction property which we
  state in the following form (see Theorem 1.1 of \cite{cavmon}) :
  
  \emph{Let $\mu_0\in\mathcal{P}(X)$ be the Dirac mass at some point $x_0$
  and $\mu_1=\tfrac{v_g}{v_g(A)}|_{A}$ for some measurable $A\subset
  X$. There exists a unique optimal dynamical transport plan
  $\Pi\in\mathcal{P}(\mathcal{G}(X))$ such that
  $(e_0)_*\Pi=\mu_0$ and $(e_1)_*\Pi=\mu_1$. Moreover, for every
  $t\in (0,1]$, $(e_t)_*\Pi$ is $v_g$-absolutely continuous and
  satisfy:
  \[(e_t)_*\Pi\leq \frac{C}{t^N}\frac{v_g}{v_g(A)}.\]
  where $C$ is a constant depending only on $K$, $N$ and $\diam(X)$
   and $e_t:\mathcal{G}(X)\to X$ maps $\gamma$ to $\gamma(t)$. }

  Let us fix $\mu_0=\delta_{x_0}$ with $x_0\in X^{reg}$,
  $\mu_1=\tfrac{v_g}{v_g(X)}$ and $\Pi\in\mathcal{P}(\mathcal{G}(X))$
  the optimal 
  dynamical transport plan defined above. Set
  $\Gamma=\{\gamma\in\mathcal{G(X)}|\exists t\in 
  [0,1]\ \gamma(t)\in\Sigma\}$. To show that $X^{reg}$ is $v_g$-almost
  everywhere convex, we need to show that $v_g(e_1(\Gamma))=0$.

  We first show that $\Pi(\Gamma)=0$. Pick $l\in\mathbb{N}$ big enough
  such that $\varepsilon\diam(X)\leq d_g(x_0,\Sigma)=\delta$  with
  $\varepsilon=2^{-k}$ and set, for every
  integer $i$ between $0$ and $2^{l}-1$ :
  \[\Gamma_i=\left\{\gamma\in\mathcal{G}(X)\middle|
      \gamma\left([\tfrac{i}{2^l},\tfrac{i+1}{2^l}]\right)\cap 
      \Sigma\neq\emptyset\right\}.\]
  We have that $\Gamma=\cup_{i=1}^{2^{l}-1}\gamma_i$. 

  Note that
  $\Gamma_i=\emptyset$ as long as $\frac{i}{2^l}\leq \delta$. 

  If
  $\Pi(\Gamma)>0$, then there are some $i$ such that
  $\Pi(\Gamma_i)>0$. For every such $i$, let
  $\Pi_i=\tfrac{\Pi}{\Pi(\Gamma_i)}|_{\Gamma_i}$. Being the
  restriction of $\Pi$, $\Pi_i$ is an optimal dynamical dynamical
  transport plan between $\delta_{x_0}$ and the borel set
  $B_i=e_1(\Gamma_i)$, thus we can apply the measure contraction
  property 
  to $\Pi_i$ and get :
  \[(e_t)_*(\Pi_i)\leq \frac{C}{t^N}\frac{v_g}{v_g(B_i)}
    =\frac{C}{t^N}\frac{v_g}{\Pi(\Gamma_i)}.\]
  Now we notice that if $\gamma\in\Gamma_i$ and $t\in
  [\tfrac{i}{2^l},\tfrac{i+1}{2^l}]$, then $d(\gamma(t),\Sigma)\leq
  D \varepsilon$ which can be rephrased as $e_t(\Gamma_i)\subset
  \Sigma^{\varepsilon}$ where $\Sigma^{\varepsilon D}$ is the
  $\varepsilon D$
  tubular neighborhood around $\Sigma$. Hence, for $t\in
  [\tfrac{i}{2^l},\tfrac{i+1}{2^l}]$ :
  \[1=(e_t)_*\Pi_i(X)= (e_t)_*(\Pi_i(\Sigma^{\varepsilon D}))\leq 
    \frac{C}{t^N}\frac{v_g(\Sigma^{\varepsilon D})}{\Pi(\Gamma_i)}\leq
    \frac{C}{(\tfrac{\delta}{D})^{N}}\frac{v_g(\Sigma^{\varepsilon
        D})}{\Pi(\Gamma_i)}\] 
  since $t\geq \tfrac{\delta}{D}$ if $\Gamma_i\neq\emptyset$.

  Thus :
  \[\Pi(\Gamma_i)\leq C\delta^{-N}D^Nv_g(\Sigma^{\varepsilon D}).\]

  We can now estimate :
  \begin{align*}
  \Pi(\Gamma)&\leq \sum_i\Pi(\Gamma_i)\leq 2^l
    C\delta^{-N}D^Nv_g(\Sigma^{\varepsilon D})\\
             &\leq 
    2^l C\delta^{-N}D^NA(\varepsilon D)^2\leq
               2^{-l}CA\delta^{-N}D^{N+2} 
  \end{align*}
  since the volume of $\Sigma^{\varepsilon}$ can be bounded from above
  by $A\varepsilon^2$ for some constant $A>0$. This comes from the
  fact that $\Sigma$ has codimension at least $2$ using the same
  ideas as section \ref{volume-measure}.

  Since $l$ can be chosen to be arbitrarily large, we
  have shown that $\Pi(\Gamma)=0$.

  Now $v_g(e_1(\Gamma))=\Pi(\Gamma)$ and thus $(X,d_g)$ is $v_g$-almost
  everywhere convex. 
\end{proof}
\begin{remark}
We actually proved here that for any subset $Y$ of an $RCD(K,N)$ space $(X,d,m)$ such that $\tfrac{m(Y^\varepsilon)}{\varepsilon}$ goes to $0$ as $\varepsilon$ goes to $0$, $X\backslash Y$ is $m$-almost everywhere convex. This applies in particular for a stratified space with singular Ricci curvature bounded below and $Y = \Sigma$ its singular set, thanks to the fact that the singular set has codimension smaller or equal than two. 
\end{remark}

We can now use the previous result together with the theorem of N. Li
\cite{NanLi} to characterize stratified spaces with
curvature bounded from below in the sense of Alexandrov (Corollary
\ref{MainCor} in the introduction). Recall that a
geodesic space $(X,d)$ is said to be an Alexandrov space with
curvature bounded from below by 
$k$ ($\CBB(k)$ in short) if geodesic triangles in $(X,d)$ are
larger than their couterparts in the simply connected surface of
constant curvature $k$. For a precise definition we refer to
\cite{bbi}, Chapters 4 and 10.

\begin{cor}
  Let $(X,g)$ be a compact stratified space. Then $(X,d_g)$ is $\CBB(k)$ if and
only if the following two conditions are satisfied :
\begin{enumerate}[(i)]
\item The sectional curvature of $g$ is larger than or equal to $k$ on $X^{reg}$.
\item The angle $\alpha$ along the codimension 2 stratum
  $\Sigma^{n-2}$ is at most $2\pi$. 
\end{enumerate}
\end{cor}

\begin{proof}
  The ``only if'' part is proven along the same lines as the
  $RCD(K,N)$ case, see section \ref{RCDversRic}. The condition on the
  regular set comes from the existence of convex neighborhoods and the
  Riemannian Toponogov Theorem. The angle condition on the codimension 2
  stratum comes the fact that tangent cones to $\CBB(k)$ spaces are
  $\CBB(0)$ spaces and that a $2$ dimensional metric cone 
  is $\CBB(0)$ if and only if its angle is at most $2\pi$.

  For the ``if'' part, we use Corollary 0.1 of \cite{NanLi}. It states
  that if in a geodesic metric space $(X,d)$ of Hausdorff dimension
  $n$, one can find an open 
  dense set $Y$ which is $\mathcal{H}^n$-almost everywhere convex and
  such that any point in $Y$ has a convex neighborhood which is $\CBB(k)$, then
  $(X,d)$ is $\CBB(k)$. In our case, $X^{reg}$ is open and dense in
  $(X,d)$, and is
  almost everywhere convex by the previous proposition. Furthermore
  every point in $X^{reg}$ has a convex neighborhood by section
  \ref{RCDversRic} which is $\CBB(k)$ by the classical Toponogov Theorem.

  Hence $(X,d_g)$ has curvature bounded from below by $k$ in the sense of
  Alexandrov.

\end{proof}

%

%% file: mainproof.tex
\section{Proof of the main theorem}
This section is devoted to the proof of our main theorem: \\

\noindent \textbf{Theorem A.} A compact stratified space $(X,d_g,v_g)$ endowed with an iterated edge metric $g$ satisfies the $RCD(K,N)$ condition if and only if its dimension is smaller than or equal to $N$ and the iterated edge metric $g$ has singular Ricci curvature bounded below by $K$ in the sense of Definition \ref{RicciBD}. \\

The proof is divided in two parts. In the first we prove that a compact stratified spaces which is also $RCD(K,N)$ has a singular Ricci lower bound. In the second part, we prove the reverse implication, by showing the Barky-Émery inequality. At the end of Section 2, we observed that a compact stratified space meets the assumption of Theorem \ref{th:eks}, and as a consequence the Bakry-Émery inequality implies that the space is an $RCD^*(K,N)$ space. The equivalence between $RCD^*(K,N)$ and $RCD(K,N)$ proven in \cite{CavallettiMilman} allows us to conclude. 

\subsection{$RCD$ implies singular Ricci curvature bounded below}
\label{RCDversRic}
\begin{proposition}
Let $X$ be an $n$-dimensional stratified space, and let $g$ be an iterated edge metric.
Assume $(X,d_g,v_g)$ satisfies the condition $CD(K,N)$ (or the condition $CD^*(K,N)$) with $K\in \mathbb{R}$ and $N\in [1,\infty)$. Then, $g$ has singular Ricci curvature bounded from below by $K$
(in the sense of a stratified space) and $n\leq N$.
\end{proposition}
\begin{proof}

{\textbf 1.}
First, the condition $CD(K,N)$ (or the condition $CD^*(K,N)$) implies that $\dim_{\mathcal{H}}\leq N$ by \cite[Corollary 2.5]{stugeo2}. Hence, $\dim_{X^{reg}}=n\leq N$.
\\

{\textbf 2.}
Moreover, consider $x\in X^{reg}$. Recall that $g$ is a Riemannian metric on $X^{reg}$ that induces a distance function $\tilde{d}$ on $X^{reg}$.
In general, it is clear that $\tilde{d}\geq d_g|_{X^{reg}\times X^{reg}}$, and $\epsilon$-balls w.r.t. $d_g$ coincide with $\epsilon$-balls w.r.t. $\tilde{d}$ 
provided $\epsilon>0$ is sufficiently small. Moreover, by \cite[Chapter 3]{docarmo} for any such $\epsilon$ we can find
$\eta\in (0,\epsilon/4)$ such that $\overline{B_{\eta}(x)}$ is geodesically convex w.r.t. $\tilde{d}$. 
\medskip

{\it Claim:
We have $\tilde{d}|_{\overline{B_{\eta}(x)}\times \overline{B_{\eta}(x)}}=d_g|_{\overline{B_{\eta}(x)}\times \overline{B_{\eta}(x)}}$.}

Indeed, if $\gamma$ is a minimizing $d_g$-geodesic between $y,z\in B_{\eta}(x)$, by the triangle inequality
we have that $\mbox{Im}\gamma\subset B_{\epsilon}(x)$. Therefore, $\tilde{d}(y,z)\leq d_g(y,z)$ since $\gamma$ is an admissible competitor for $\tilde{d}$. The other inequality already holds, and therefore the claim follows.
\smallskip

Hence, $(Y,\tilde{d}=d_g|_{Y},v_g|_{Y})$ with $Y=\overline{B_{\eta}(x)}$ is a geodesically convex subspace with positive measure of a metric measure space $(X,d_g,v_g)$ satisfying the condition $CD(K,N)$,
and therefore satisfies the condition $CD(K,N)$ as well 
(or the condition $CD^*(K,N)$) by \cite[Proposition 1.4]{stugeo2}.

Then, we can procede with similar arguments as in the proofs of Theorem 1.7 in \cite{stugeo2}, Theorem 1.1 in \cite{renessesturm}, or Theorem 7.11 in \cite{agguide}.
We note that one usually assumes the context of a closed Riemannian manifold without boundary that is different from ours. But 
it is clear that the arguments adapt to the case of an open, geodesically convex domain. For instance, let us briefly outline the argument from \cite{agguide} (compare also with the proof Theorem 6.1 in \cite{ketterermondino}).

Assume the condition $CD(K,N)$ holds but there exists a regular point $x\in X$ and a tangent vector $v$ at $x$ such that $\ric_g|_x(v)\leq (K-4\epsilon)|v|^2$. Then, one can pick $\eta$ as above, and one finds a 
smooth function $\phi$ with compact support in $B_{\eta}(x)$ such that 
$$
\nabla \phi|_x=v\ \ \ \&\ \ \ \nabla^2\phi(x)=0. 
$$
We can replace $\phi$ and $v$ by $\delta\phi$ and $\delta v$ such that the previous remains true and $\phi$ becomes a smooth Kantorovich potential. If we define 
$T_t(y)=\exp_y(-t\nabla\phi|_y)$ for $t\in [0,1]$ and $$\mu_t=T_{\star}\mu_0 \ \ \mbox{ with }\ \ \mu_0=v_g(B_{\theta}(x))^{-1}v_g|_{B_{\theta}}(x).$$ then $t\in [0,1]\mapsto \mu_t$ becomes a smooth $L^2$-Wasserstein geodesic.
Note that by choice of $B_{\eta}(x)$ and $\phi$ each transport geodesic $t\in [0,1]\mapsto T_t(y)$ is contained in $B_{\eta}(x)$.
By choosing $\delta$ and $\theta$ sufficiently small one can achieve that no transport geodesic meets a cut point and $\sigma_y:t\in [0,1]\mapsto \log\det DT_t(x)$ satisfies
$$
\sigma''_y + \frac{1}{n}\left(\sigma_y'\right)^2+ K-\epsilon\geq 0 \ \mbox{ on } [0,1].
$$
for any $y\in B_{\eta}(x)$. The previous Riccatti-type inequality in particular follows from smooth Jacobi field compuations for geodesic variations in $B_{\eta}(x)$. From this
one can deduce an inequality for $S_n$ along $(\mu_t)_{t\in [0,1]}$ like in the definition of $CD(K,N)$ but with reverse inequalities and $K$ replaced by $K-\epsilon$ (again compare with \cite{ketterermondino}). 
This gives a contradiction.  
\\

{\textbf 3.}
Pick a point $x\in \Sigma^{n-2}$. 
Note that the corresponding $Z_x^1$ satisfies $Z_x^{n-2}=Z_x^{n-2,reg}\simeq \mathbb{S}^1$ with $h=c^2(d\theta)^2$ for some {$c\in (0,+\infty)$} and the standard metric $(d\theta)^2$.
From the exposition in subsection \ref{subsubsec:tangentcones} we have that $(B_{1/n}(x),nd_g)$ converges in 
Gromov-Hausdorff sense to $(C(S_x^{n-1}),d_{C})$ that is the metric euclidean cone over $S^{n-1}_x$ which is the $(n-2)$-fold spherical suspension of the link $Z_x^{n-2}$. 
Hence, since $(C(S_x^{n-1}),d_{C})$ is the (measured) Gromov-Hausdorff limit of a sequence of $CD(\frac{1}{n}K,N)$-spaces ($CD^*(\frac{1}{n}K,N)$-spaces respectively), it is an $CD(0,N)$-space itself.

Finally, we can apply \cite[Corollary 2.6]{bastco} that yields $\diam{S_x^{n-1}}\leq \pi$. Note that Corollary 2.6 in \cite{bastco} is a result about the euclidean cone over some one dimensional 
manifold but it is clear from the proof that the statement holds as well in our context. Since by definition of $n$-fold spherical suspensions a copy of $Z_x^{n-2}$ is isometrically 
embedded into $S_x^{n-1}$, the 
link $Z_x^{n-2}$ has bounded diameter by $\pi$ as well.
\end{proof}

\subsection{Singular Ricci curvature bounded below implies $RCD$}

For the second implication, we are going to prove that a compact stratified space with a singular lower Ricci curvature bound satisfies the Bakry-Émery condition given in Definition \ref{def:BE}. As illustrated in Subsection 2.1.5, we can then apply Theorem \ref{th:eks} and conclude. For the sake of clarity, we state here the weak Bochner inequality of Definition \ref{def:BE} in the setting of a compact stratified space. 

A compact stratified space $(X,d_g,v_g)$ satisfies the $BE(K,N)$ condition for $K \in \R$ and $N \in \N$ if for any function $u \in W^{1,2}(X)$ such that $\Delta_g u \in W^{1,2}(X)$ and for any test function $\psi \in W^{1,2}(X) \cap L^{\infty}(X)$ such that $\Delta_g\psi \in L^{\infty}(X)$, $\psi \geq 0$ we have:
\begin{equation}
\label{BE}-
\frac 12 \int_X \Delta_g \psi |du|^2 dv_g +\int_X \psi (\nabla(\Delta_g u), \nabla u)_g dv_g \geq \int_X \psi\left(K|du|^2+\frac{(\Delta_g u)^2}{N}\right)dv_g. 
\end{equation}

\subsubsection{Proof of the Bochner inequality on stratified spaces}

We are going to prove the weak Bochner inequality first for an eigenfunction $\varphi$ of the Laplacian, then for finite linear combinations of eigenfunctions. Since eigenfunctions are dense in the domain of the Laplacian, we will get a first Bochner inquality with the further assumption that the test function $\psi$ has bounded gradient. We will be able to drop this assumption and get the inequality \ref{BE} by using the ultracontractivity of the heat semigroup. 

We recall here some regularity properties of eigenfunctions. First of all, we know that an eigenfunction $\varphi$ belongs to $W^{1,2}(X)\cap L^{\infty}(X)$; moreover, when we have singular Ricci curvature bounded below by a constant 
$K\in \R$ as in Definition \ref{RicciBD}, it is possible to show that 
an eigenfunction belongs to $W^{2,2}(X)$ and that it's gradient is bounded. This is proven in \cite{M14}  when $K=(n-1)$, but actually does not depend on $K$ being positive. We sketch briefly the main lines of the proof without the assumption that $K$ is positive. Corollary 2.4 in \cite{TM} states the following: 

\begin{proposition}
\label{logTub}
Let $X$ be a compact stratified space of dimension $n$, endowed with an iterated edge metric $g$.  Let $\varphi$ be an eigenfunction for the Laplacian and $\Sigma^{\eps}$ a tubular neighbourhood of the singular set $\Sigma$ of size $\eps >0$. 
Assume that for any $x \in X$, the tangent sphere $S_x$ is such that $\lambda_1(S_x)\geq (n-1)$. Then there exists a positive constant $C$ such that
\begin{equation*}
\norm{\nabla \varphi}_{L^{\infty}(X \setminus \Sigma^{\eps})}\leq C \sqrt{|\ln(\eps)|}.
\end{equation*}
\end{proposition}

If the iterated edge metric $g$ is such that the singular Ricci curvature is bounded below by $K \in \R$, as in Definition \ref{RicciBD}, then the assumption of the previous Proposition holds. 
Indeed, we know that $\ric_g \geq K$ on $X^{\tiny{reg}}$ implies that the singular Ricci curvature of each link $(Z_j,k_j)$ is bounded below by $(\mbox{dim}(Z_j)-1)$. 
As a consequence, when $x$ belongs to a stratum of codimension larger than two, the tangent sphere $(S_x,h_x)$ has singular Ricci curvature bounded below by $(n-2)$.
Then the Lichnerowicz theorem in \cite{M14} implies that $\lambda_1(S_x)\geq (n-1)$. Since we also assumed that the angles along the stratum of
codimension 2 are smaller than $2\pi$, we have the same lower bound for $\lambda_1(S_x)$ when $x$ belongs to $\Sigma^{n-2}$. Therefore, we can apply Proposition \ref{logTub} and get the estimate on the gradient of eigenfunctions. 

In the proof of Lichnerowicz theorem in \cite{M14}, we also deduce that $|\nabla \varphi|$ belongs to $W^{1,2}(X)\cap L^{\infty}(X)$. This is done by using the Bochner inequality on the regular set and by constructing the appropriate family of cut-off functions:

\begin{lemma}
\label{cutoff}
Let $X$ be a stratified space and $g$ an iterated edge metric with singular Ricci curvature bounded below by $K \in \R$. 

Then for any $\varepsilon >0$ there exists a family of cut-off functions $\cutoff \in C^{\infty}_0(X^{\tiny{reg}})$, which satisfy the following properties:
\begin{itemize}
\smallskip
\item[1.] $0 \leq \cutoff \leq 1$ and $\cutoff$ vanishes on a tubular neighbourhood of the singular set;
\smallskip
\item[2.] The norm in $L^2(X)$ of $|\nabla\cutoff|$ and the norm in $L^1(X)$ of $|\Delta_g\cutoff|$ converge to zero when $\eps$ tends to zero. 
\end{itemize} 
\end{lemma}
 
\noindent This argument does not depend on $K$ being positive. For the details of the construction, see \cite{M14} or \cite{TM}. We summarize in the following: 

\begin{proposition}
Let $X$ be an $n$-dimensional stratified space endowed with the iterated edge metric $g$. Assume that $g$ has singular Ricci curvature bounded below by $K \in \R$. Then any eigenfunction $\varphi$ of the Laplacian belongs to $W^{2,2}(X)\cap L^{\infty}(X)$ and the gradient $\nabla \varphi$ has bounded norm on $X$.
\end{proposition}

In particular, eigenfunctions are Lipschitz functions for a compact stratified space satisfying Definition \ref{RicciBD}. Note that this approach does not apply in presence of angles larger than $2\pi$ along the stratum of codimension two. Indeed, when the angles are larger than $2\pi$, Theorem A in \cite{ACM14} implies that eigenfunctions are at most $\beta$-Hölder continuous with $\beta <1$.

The regularity of eigenfunctions and the existence of an appropriate family of cut-off functions allows us to prove the Bochner inequality for an eigenfunction $\phi$:

\begin{proposition}[Bochner inequality for eigenfunctions]
Let $(X^n,g)$ be a stratified space, whose iterated edge metric $g$ has singular Ricci curvature bounded below by $K \in \R$. 

Then for any $\varphi$ eigenfunction of the Laplacian $\Delta_g$ and $\psi \in \mathcal{D}(\Delta_g)\cap L^{\infty}(X)$ such that $\Delta_g \psi \in L^{\infty}(X)$ we have:
\begin{equation}
\label{Bochner1}-
\frac 12 \int_X \Delta_g \psi |d\varphi|^2 dv_g +\int_X \psi (\nabla(\Delta_g \varphi), \nabla \varphi)_g dv_g \geq \int_X \psi\left(K|d\varphi|^2+\frac{(\Delta_g\varphi)^2}{n}\right)dv_g. 
\end{equation}
\end{proposition}

\begin{proof}
Since $\Delta_g \varphi=\lambda \varphi$, and since the Bochner formula holds on the regular set $X^{\tiny{reg}}$ we have:
\begin{equation*}
-\frac{1}{2}\Delta_g |d\varphi|^2+\lambda|d\varphi|^2= \ric_g(d\varphi,d\varphi)+|\nabla d\varphi|^2 \ \mbox{ on }\  X^{\tiny{reg}}.
\end{equation*}
Note that $\varphi$ is smooth on the regular set $X^{\tiny{reg}}$, and therefore $\Delta_g|d\varphi|^2$ is well-defined on $X^{\tiny{reg}}$.

Now consider $\psi \in \mathcal{D}(\Delta_g)\cap L^{\infty}(X)$ such that $\Delta_g \psi \in L^{\infty}(X)$ and for $\eps >0$ choose a cut-off function $\rho_{\eps}$, $0 \leq \cutoff \leq 1$, vanishing on a tubular neighbourhood of the singular set, as in Lemma \ref{cutoff}. We multiply the previous equality by $\cutoff \psi$ and then integrate on $X$: 
\begin{align*}
-\frac{1}{2}\int_X \cutoff \psi \Delta_g |d\varphi|^2 dv_g +\int_X \cutoff \psi \lambda |d\varphi|^2 dv_g & =\int_X \cutoff \psi (\ric_g(d\varphi, d\varphi)+|\nabla d\varphi|^2) dv_g.
\end{align*}
As for the right-hand side, we use-Cauchy-Schwarz inequality and and the fact that $\ric_g\geq K$ on the regular set in order to get:
\begin{equation*}
\int_X \cutoff \psi (\ric_g(d\varphi, d\varphi)+|\nabla d\varphi|^2) dv_g \geq \int_X \cutoff \psi  \left(K|d\varphi|^2+\frac{(\Delta_g \varphi)^2}{n}\right)dv_g.
\end{equation*}
This converges to the right-hand side of \ref{Bochner1} when $\eps$ goes to zero. As for the second term in the left-hand side, we have:
\begin{equation*}
\int_X \cutoff \psi \lambda |d\varphi|^2 dv_g=\int_X \cutoff \psi (\nabla(\Delta_g\varphi), \nabla \varphi)_g dv_g, 
\end{equation*}
which also converges to the second term in the left-hand side of the Bochner inequality \ref{Bochner1}. It remains to study the first term in the left-hand side. By integrating by parts we obtain:
\begin{align}
\label{IPP}
\int_X \cutoff \psi \Delta_g |d\varphi|^2 dv_g  =&\int_X \cutoff \Delta_g\psi |d\varphi|^2 dv_g+\int_X  \psi \Delta_g\cutoff |d\varphi|^2\\
&\ \ \ \ \ - 2\int_X (d\psi, d\cutoff)_g |d\varphi|^2 dv_g \nonumber . 
\end{align}
The first term in the right hand side in this last identity converges to the first term in the right-hand side of \ref{Bochner1} when $\eps$ goes to zero, then we need to show that the other two terms tend to zero as $\eps$ goes to zero. Consider the second term in the right-hand side of \ref{IPP}. Since $\psi$ and $|d\varphi|$ belong to $L^{\infty}(X)$ we have:
\begin{equation*}
\left|\int_X  \psi \Delta_g\cutoff |d\varphi|^2 dv_g \right| \leq c\int_X|\Delta_g \cutoff| |d\varphi|^2 dv_g \leq c_1 \int_X |\Delta_g \cutoff| dv_g.
\end{equation*}
Now, $\cutoff$ is constructed in such a way that this last integral converges to zero as $\eps$ goes to zero. As for the last term in \ref{IPP} we can again use that $|d\varphi|$ is bounded and the Cauchy-Schwarz inequality in order to get:
\begin{equation*}
\left| \int_X (d\psi, d\cutoff)_g |d\varphi|^2 dv_g \right| \leq c_1 \left(\int_X |d\psi|^2 dv_g \right)^{\frac 12}\left(\int_X |d\cutoff|^2 dv_g\right)^{\frac 12},
\end{equation*} 
and $\cutoff$ is chosen in such a way that the norm of its gradient in $L^2(X)$ tends to zero as $\eps$ goes to zero. As a consequence, we get the desired Bochner inequality. 
\end{proof}


\begin{proposition}[Finite linear combinations] Under the same assumptions on $X$, $g$ and $\psi$, consider a finite linear combination of eigenfunctions: 
$$\varphi=\sum_{k=1}^N a_k \varphi_k.$$ 
Then the Bochner inequality \ref{Bochner1} holds for $\varphi$. 
\end{proposition}

\begin{proof}
Observe that $\varphi$ has the same regularity as an eigenfunction, meaning that $\varphi$ belongs to $\mathcal{D}(\Delta_g) \cap L^{\infty}(X)$, its Laplacian $\Delta_g\varphi$ and gradient $|d\varphi|$ are bounded, and it is smooth on $X^{\tiny{reg}}$. Moreover, the Bochner formula holds on the regular set $X^{\tiny{reg}}$; we have then:
$$-\frac 12 \Delta_g |d\varphi|^2+(\nabla(\Delta_g \varphi), \nabla\varphi)_g=\ric_g(d\varphi,d\varphi)+|\nabla d\varphi|^2 \ \mbox{ on }\ X^{\tiny{reg}}.$$
As we did before, we multiply this equality by $\cutoff \psi$ and integrate on $X$:
\begin{align}
\label{Bochner2}
&-\frac 12 \int_X \cutoff \psi \Delta_g|d\varphi|^2 dv_g + \int_X \cutoff \psi (\nabla(\Delta_g \varphi), \nabla \varphi)_g dv_g \nonumber\\
&\hspace{5cm} = \int_X \cutoff \psi \left(\ric_g(d\varphi,d\varphi)+|\nabla d\varphi|^2 \right) dv_g
\end{align}
The right-hand side of this equality is bounded by below by: 
\begin{equation*}
\int_X \cutoff \psi \left( K|d\varphi|^2+\frac{(\Delta_g \varphi)^2}{n} \right) dv_g, 
\end{equation*}
which converges to the right-hand side of the desired Bochner inequality when $\eps$ tends to zero. The second term in the left-hand side of \ref{Bochner2} also converges to the corresponding term in the Bochner inequality, since all the quantities playing here are bounded. It remains to study the first term in the left-hand side of \ref{Bochner2}. We decompose it as before by integrating by parts; since $|d\varphi|$ is bounded, we can apply the same argument as before to get that, when $\eps$ goes to zero, the first term in the left-hand side of \ref{Bochner2} tends to: 
$$\int_X \Delta_g\psi |d\varphi|^2dv_g.$$
This concludes the proof and proves that the Bochner inequality holds for finite linear combinations of eigenfunctions. 
\end{proof}

\begin{proposition}
Let $X$ be a compact stratified space of dimension $n$, endowed with iterated edge metric $g$ with singular Ricci curvature bounded below by $K \in \R$.

\noindent Then for all functions $\phi \in \mathcal{D}(\Delta_g)$ with $\Delta_g \phi \in W^{1,2}(X)$ and all $\psi\in \mathcal{D}(\Delta_g)\cap L^{\infty}(X)$, with $\psi \geq 0$, bounded gradient $|\nabla \psi|$ and Laplacian $\Delta_g \psi$, we have
\begin{equation*}
-\frac 12 \int_X \Delta_g \psi |d\phi|^2 dv_g +\int_X \psi (\nabla(\Delta_g \phi), \nabla \phi)_g dv_g \geq \int_X \psi\left(K|d\phi|^2dv_g +\frac{(\Delta_g\phi)^2}{n}\right)dv_g. 
\end{equation*}
\end{proposition}

\begin{proof}
Denote by $\{\lambda_i\}_{i \in \N}$ the sequence of eigenvalues of the Laplacian $\Delta_g$, define $V=\mbox{span}\{\varphi_i\}_{i \in \N}$ and the multiplication operators $L_i$ on $V$ by: 
$$L_i u= a_i\lambda_i \varphi_i, \quad \mbox u=\sum_{k \in \N} a_k\varphi_k.$$
Consider the operator
$$L=\bigoplus_{i \in \N} L_i,$$
which is essentially self-adjoint and closable (see Problem 1(a) in Chapter X, \cite{ReedSimon}). Observe that the Laplacian $\Delta_g$ is a self-adjoint extension of $L$, thus it is its unique self-ajdoint extension. 

\noindent We can also construct self-adjoint extensions by considering the Friedrichs extenstion $L_F$ and the closure $\bar{L}$ of $L$. The first is obtained as the self-adjoint operator whose domain is the closure of $V$ with the norm:
$$||u||_F^2=||u||_2^2+(Lu,u)=||u||_2^2+\int_X u\Delta_g u dv_g=||u||^2_{1,2}.$$
As for the second, one needs to close $V$ with respect to the graph norm: 
\begin{equation}
\label{graphnorm}
||u||_L^2=||u||_2^2+||Lu||_2^2= ||u||_2^2+\int_X (\Delta_g u)^2 dv_g.
\end{equation}
Since $L$ is essentially self-adjoint and its extension is the Laplacian, these two extensions coincides and $V$ is dense in $\mathcal{D}(\Delta_g)$ with 
respect to both the norm of $W^{1,2}(X)$ and the graph norm \ref{graphnorm}. Therefore, for each function $\phi$ in the domain of the Laplacian, 
there exists a sequence $\{u_i\}_{i \in \N} \subset V$ which converges to $\phi$ in $W^{1,2}(X)$ and such that $\{\Delta_g u_i\}_{i \in N}$ converges to $\Delta_g \phi$. Since the Bochner inequality holds for any $u_i \in V$, we have:
\begin{equation*}
-\frac 12 \int_X \Delta_g \psi |du_i|^2 dv_g +\int_X \psi (\nabla(\Delta_g u_i), \nabla u_i)_g dv_g \geq \int_X \psi\left(K|du_i|^2dv_g +\frac{(\Delta_gu_i)^2}{n}\right)dv_g. 
\end{equation*}
We can pass to the limit as $i$ goes to infinity in the right-hand side and in the first term of the left-hand side, since both $\psi$ and its Laplacian are bounded. 
As for the second term in the left-hand side, we can rewrite it in the following way: 
\begin{equation*}
\int_X \psi (\nabla(\Delta_g u_i), \nabla u_i)_g dv_g = \int_X \psi (\Delta_g u_i)^2 dv_g - \int_X \Delta_g u_i (\nabla u_i, \nabla \psi)_g dv_g. 
\end{equation*}
Since $|\nabla \psi|$ is bounded we can use Cauchy-Schwarz inequality twice to get: 
$$
\int_X \Delta_g u_i (\nabla u_i, \nabla \psi)_g dv_g \leq C ||\Delta_g u_i||_2||\nabla u_i||_2. 
$$
Therefore, when we pass to the limit as $i$ goes to infinity we get: 
$$
\int_X \psi (\Delta_g u_i)^2 dv_g - \int_X \Delta_g u_i (\nabla u_i, \nabla \psi)_g dv_g \rightarrow \int_X \psi (\Delta_g \phi)^2dv_g -\int_X \Delta_g \phi(\nabla \phi, \nabla \psi)_g dv_g.
$$
As a consequence we can pass to the limit in the second term of the left-hand side of the Bochner inequality, and we get the desired inequality. 
\end{proof}

In order to have the integral Bochner inequality \ref{BE} implying $\mbox{RCD}(K,n)$, for the right set of test functions, we need to drop the assumption that the test function has bounded gradient. In order to do that, we are going to use the properties of the heat semigroup that we recalled in the first section. 

We are now in position to prove:

\begin{theorem}
Let $X$ be an n-dimensional stratified space endowed with an iterated edge metric $g$ with singular Ricci curvature bounded below by $K\in \R$. Then $X$ satisfies the Bakry-Émery condition $BE(K,n)$. 
\end{theorem}

\begin{proof}
Consider $\psi$ a test function such that $\psi \in \mathcal{D}(\Delta_g)\cap L^{\infty}(X)$, $\psi\geq 0$ and $\Delta_g \psi$ is bounded. Up to adding a positive constant to $\psi$, 
we can assume that $\psi$ is strictly positive. Let $P_t$ be the heat semigroup associated to the Laplacian; since $V$, the span of eigenfunctions, is dense in the domain of the Laplacian, 
let $\psi_i$ a sequence in $V$ converging to $\psi$ in $W^{1,2}(X)$ with $\Delta_g \psi_i$ converging in $L^2(X)$ to $\Delta_g \psi$. For fixed $t>0$, consider $P_t \psi_i$. Because of ultracontractivity of the heat semigroup in Lemma \ref{ultracontractive}, we have 
$$||P_t(\psi -\psi_i)||_{\infty} \leq C_t||\psi-\psi_i||_{2}$$
for any $t>0$, 
and therefore $P_t\psi_i$ uniformly converges to $P_t\psi$; since $\psi$ is positive, so is $P_t \psi$, and then for $i$ large enough $P_t\psi_i$ is positive too. Moreover, $\psi_i$ is a finite linear combination of eigenfunctions, then $P_t \psi_i$, $\nabla P_t\psi_i$ and $\Delta_g P_t \psi_i$ all belongs to $L^{\infty}(X)$. As a consequence, $P_t \psi_i$ satisfies the assumptions of the previous theorem and we can use it as a test function in the Bochner inequality: for all $u \in \mathcal{D}(\Delta_g)$ with $\Delta_g u \in W^{1,2}(X)$ we have
\begin{align*}
&-\frac 12 \int_X \Delta_g P_t\psi_i |du|^2 dv_g +\int_X P_t\psi_i (\nabla(\Delta_g u), \nabla u)_g dv_g \\
&\hspace{5cm}\geq \int_X P_t\psi_i\left(K|du|^2+\frac{1}{n}(\Delta_g u)\right)dv_g. 
\end{align*}
Using the uniform convergence, we can pass to the limit as $i$ goes to infinity and get: 
\begin{equation*}
-\frac 12 \int_X \Delta_g P_t\psi |du|^2 dv_g +\int_X P_t\psi (\nabla(\Delta_g u), \nabla u)_g dv_g \geq \int_X P_t\psi\left(K|du|^2+\frac{(\Delta_g u)^2}{n}\right)dv_g. 
\end{equation*}
Now if we consider the limit as $t$ goes to zero, we know that for any bounded function $f$, $P_t f$ converges to $f$ w.r.t. weak-*-topology in $(L^1(X))^*$. We can use this with $f=\psi$ and since in the previous inequality $P_t\psi$ is multiplied by functions belonging to $L^1(X)$, we can pass to the limit for $t$ going to zero and obtain: 
\begin{equation*}
-\frac 12 \int_X \Delta_g \psi |du|^2 dv_g +\int_X \psi (\nabla(\Delta_g u), \nabla u)_g dv_g \geq \int_X\psi\left(K|du|^2+\frac{(\Delta_g u)^2}{n}\right)dv_g. 
\end{equation*}
as we wished. 
\end{proof}

\begin{cor}
Let $X$ be an n-dimensional stratified space endowed with an iterated edge metric $g$ with singular Ricci curvature bounded below by $K \in \R$. Then $X$ satisfies $RCD(K,N)$ for any $N \geq n$. 
\end{cor}

\noindent Indeed, Proposition 4.9 and Theorem 4.19 in \cite{EKS} state that the Bakry-Émery condition $BE(K,N)$ implies $RCD^*(K,N)$. Then $X$ is essentially non-branching, and  \cite{CavallettiMilman} proved that for an essentially non-branching metric measure space of finite measure, $RCD^*(K,N)$ is equivalent to $RCD(K,N)$. This concludes our proof.


%% file: appendix.tex
\section*{Appendix}

\subsection*{Distance on a stratified space} 

In this part, we provide some technical facts needed to check that the length structure introduced in Section \ref{dist_strat} meets the assumptions described in \cite[Section 2.1]{bbi}, by using the local description of geodesic balls given in the first section. We also prove Lemma \ref{PropLipCurve}.

\begin{LemmaA}Let $(X,g)$ be a compact stratified space of dimension $n$ endowed with the iterated edge metric $g$.

\begin{enumerate}

\item For $x,y \in X^{\tiny{reg}}$, there exists an admissible curve $\gamma$ between $x$ and $y$ of finite length: $L_g(\gamma)<+\infty$.

\item For $x \in \Sigma$, there exists $C>0$ such that for any $r>0$ small enough, any radial curve $\rho: [0,r) \rightarrow B_{0,j}(x,r) \sim C_{[0,r)}(S_x)$ (where $\sim$ means the sets are homeomorphic) with respect to the cone metric $g_{C}=ds^2+s^2h_x$ satisfies
$$ L_g(\rho) \leq Cr.$$

\item For $x \in \Sigma$, $r>0$ small enough, there exists $\eps=\eps(r)$ such that for any admissible curve $\gamma \subset B_{0,j}(x,r) \sim C_{[0,r]}(S_x)$ from $x$ to a point $(r,y)\in C_{[0,r]}(S_x)$ satisfies
$$ L_g(\gamma) \geq (1-\eps)r.$$
\end{enumerate}

\end{LemmaA}

\begin{remark}
For $x \in X^{\tiny{reg}}$, the last two items can be proven using the existence of arbitrary small geodesically convex neighbourhoods (see for instance \cite[Chapter 3]{docarmo}) and Gauss' lemma.
\end{remark}

\begin{proof}

To prove item (1), take any continuous curve $\gamma$ from $x$ to $y$ contained in the open connected set $ X^{\mbox{\tiny{\emph{reg}}}}$; observe that any regular point admits a neighbourhood where the iterated edge metric is locally Lipschitz equivalent to the standard Euclidean metric, the compactness of the image of $\gamma$ then guarantees $\gamma$ has finite length.

To prove item (2), take $r$ so small that 
\begin{equation}\label{appen-tec}
|\psi_x^*g(\dot{\rho},\dot{\rho}) -g_{C}(\dot{\rho},\dot{\rho})|\leq \Lambda r^{\alpha}g_{C}(\dot{\rho},\dot{\rho})
\end{equation}
holds on  $B_{0,j}(x,r)^{\tiny{reg}} \sim C_{[0,r)} (S_x^{\tiny{reg}})$, thanks to \ref{g_C}. Note that $g_{C}(\dot{\rho},\dot{\rho})=1$ since $\rho(t) = (t,y) \in C_{[0,r)} (S_x^{\tiny{reg}}) $ and $g_{C}$ is a cone metric.

The proof of item (3) also builds on \ref{appen-tec}. One can assume that $r$ is so small that $\gamma(t)= (r(t),y(t)) \in C_{[0,r)} (S_x^{\tiny{reg}})$ for $t\neq 0$. Then, the above equation gives us

$$ \psi_x^*g(\dot{\gamma},\dot{\gamma}) \geq (1- \Lambda r^{\alpha})g_{C}(\dot{\gamma},\dot{\gamma}) \geq (1-\eps)\dot{r}^2(t). $$
The result follows by integrating this inequality.
 
\end{proof}

\begin{LemmaA}
\label{LipCurve}
Let $(X,g)$ be a compact stratified space of dimension $n$ endowed with the iterated edge metric $g$. Let $\gamma:[0,1] \rightarrow X$ be an admissible curve. For any $\varepsilon$ there exists an admissible curve $\gamma_{\varepsilon}$ with the same endpoints as $\gamma$ and such that $\gamma_{\eps}((0,1))$ is contained in the regular set $X^{\tiny{reg}}$ and:
$$\mbox{L}_g(\gamma_{\varepsilon}) \leq  \mbox{L}_g(\gamma) +\varepsilon.$$
\end{LemmaA} 

The proof is done by induction on the dimension of the stratified space; note that in one dimension, the only compact stratified space is the circle whose singular set is empty. Therefore, from now on we assume: \\

\noindent \textbf{Induction hypothesis:} For some $n >1$, and for any compact stratified space of dimension $(n-1)$, the previous proposition holds. \\

By definition of an admissible curve, $\gamma$ meets the singular set $\Sigma$ at most finitely many times. Therefore by additivity of the length, it suffices to prove the result in the case where $\gamma$ meets $\Sigma$ in exactly one point. The proof of this fact is in two steps. First, we prove this result in the case of an exact cone metric on a truncated cone, for a curve that only intersects the singular set at the tip of the cone. Then we will use the description of geodesic balls given in the first section: in a small ball around a singular point the iterated edge metric is close to the exact cone metric on a truncated cone over the tangent sphere, therefore we can apply a similar construction to the one given in the case of an exact cone metric. Let us start with:

\begin{LemmaA}
\label{exactCone}
Let $(S,h)$ be a stratified space of dimension $(n-1)$ and consider the metric cone $(C(S), d\rho^2+\rho^2h)$ over $S$; denote by $o$ the vertex of the cone. Let $x,y$ be two regular points in $S$, $r \in (0,1)$ and $\gamma:[-a,a] \rightarrow C(S)$ an admissible curve connecting $(r,x)$ and $(r,y)$ such that $\gamma(0)=o$ and $\gamma(t) \in C(S^{\mbox{\tiny{\emph{reg}}}})\setminus \{o\}$ for any $t\neq 0$. Then for any $\varepsilon >0$, there exists $\gamma_{\eps}: [-a,a] \rightarrow C(S)$ such that: 
\begin{itemize}
\item[(i)]$\gamma_{\varepsilon}(t)$ belongs to $C(S^{\mbox{\tiny{\emph{reg}}}})\setminus \{o\}$ for all $t \in [-a,a]$;
\item[(ii)] $L_C(\gamma_{\varepsilon}) \leq L_C(\gamma) + \varepsilon,$ where $L_C$ is the length with respect to the exact cone metric $d\rho^2+\rho^2h$.
\end{itemize}
\end{LemmaA}

\begin{proof}

Fix $\eps >0$, let $\delta \in (0,r)$ to be chosen later. By continuity of $\gamma$, there exists $t_0,t_1$ such that $t_0<0<t_1$, $\gamma(t_0)=(\delta,x), \gamma(t_1)=(\delta,y)$, and the radial coordinate $\rho (\gamma(t)) \leq \delta$ for $t\in [t_0,t_1]$; let us set $c_1 = \gamma|_{[-a,t_0]}$ and $c_2=\gamma|_{[t_1,a]}$, by hypothesis on $\gamma$, both $c_1$ and $c_2$ lie in the regular set of the cone $C(S)$. Now consider an admissible curve in the $(n-1)$-dimensional stratified space $S$ connecting $x$ and $y$. Thanks to the induction hypothesis, there exists an admissible curve $c_{\eps}$ from $x$ to $y$, lying in the regular set of $S$, and whose length in $S$ with respect to the metric $h$ satisfies $L_h(c_{\eps})\leq L_h(c)+ \eps$. Define $\gamma_{\eps}$ to be the concatenation of $c_1$, $c_{\eps}$ and $c_2$. Its length therefore satisfies: 
$$L_C(\gamma_{\eps}) \leq L_C(\gamma) +\delta L_h(c_{\eps}),$$
where we used that the length of $c_{\eps}$ with respect to the exact cone metric of $C(S)$ is $L_C(c_{\eps})=L_{\delta^2h}(c_{\eps})$. We can choose $\delta$ small enough so that $L_C(\gamma_{\eps})\leq  L_C(\gamma) +\eps.$ 
\end{proof}

Now consider a general compact stratified space $(X^n,g)$. We recall that  any $x \in X$ admits an open neighbourdood homeomorphic to the truncated cone $C(S_x)$ over the tangent sphere $S_x$.
Moreover, if we denote by $g_C=d\rho^2+\rho^2h$ the  cone metric on $C_{[0,r_0)}(S_x)$, we know thanks to \ref{g_C} that in $C_{[0,r_0)}(S_x)$ the metric $g$ is not far from $g_C$: there exists positive constants $\Lambda$  and $\alpha$ such that
\begin{equation}\label{e-append-lip}
|\psi^*_x  g-g_C|<\Lambda r_0^{\alpha}.
\end{equation}

Consider an admissible curve $\gamma$ of finite length $L_g(\gamma)$. We use the same notation and apply the same construction as the one in the proof of Lemma \ref{exactCone}. The point is to estimate the length $L_g(c_{\eps})$. Since $c_{\eps}$ is contained in $ S^{\mbox{\tiny{\emph{reg}}}}$, we can further assume the curve has constant speed $L_h$ with respect to the metric $h$. Thus, \ref{e-append-lip} yields
$$ |\psi^*_x g(\dot{c_{\eps}},\dot{c_{\eps}}) - \delta^2 L_h^2|\leq \Lambda \delta^{\alpha+2} L_h.$$

As a consequence, we get $L_g(c_{\eps})\leq \eps $ provided $\delta $ is chosen small enough. By construction of $\gamma_{\eps}$, we obtain
$$ L_g(\gamma_{\eps}) \leq L_g(\gamma) +\eps.$$

\subsection*{A construction of a singular stratum} We are going to illustrate one of the examples in Section 3. Consider a round sphere $\s^3$ with round metric $g_0$ and a closed circle $\s^1_{\beta}$ in $\s^3$. We are going to show that we can write the metric $g_0$ in a tubular neighbourhood $\mathcal{U}^{\eps}$ of $\s^1_{\beta}$ of size $\eps$ small enough so that $g_0$ is a perturbation of the product metric $dr^2+r^2d\varphi^2+ a^2d\theta^2$ (for $a$ the radius of $\s^1_{\beta}$). More precisely, we show that there exists a positive constant $\Lambda$ such that: 
\begin{equation}
\label{asymptotic}
|g_0 - (dr^2+r^2d\varphi^2+a^2d \theta^2)| \leq \Lambda r^{\gamma}+o(r^{\gamma}),
\end{equation}
with $\gamma=1$ if the circle $\s^1_{\beta}$ is not totally geodesic and $\gamma=2$ otherwise. 

We look at the sphere $\s^3$ in $\R^4=\R^2\times \C$ with coordinates $(x_1,x_2,\rho e^{i\theta})$. Up to changing the coordinate system in $\R^4$, we can parameterize  the circle $\s^1_{\beta}$ with the curve: 
\begin{equation*}
c(\theta)=
\left(
\begin{array}{c}
\cos(\beta) \\
0 \\
\sin(\beta)e^{i\theta}
\end{array}
\right). 
\end{equation*}
Observe that $\s^1_{\beta}$ is a circle of radius $a=\sin(\beta)$ contained in the totally geodesic sphere $\s^2$ in $\s^3$, obtained as the intersection of $\s^3$ with the plane $\{x_2=0\}$. Moreover, the case $\beta = \frac{\pi}{2}$ corresponds to a great circle in this $\s^2$. Then for $\beta=\pi/2$, $c$ is totally geodesic in $\s^3$. We aim to write the metric on a tubular neighbourhood $\mathcal{U}^{\eps}$ of $\s^1_{\alpha}$ as an admissible metric for a stratified space; we start by parameterizing the coordinates of a point in a tubular neighbourhood of $\s^1_{\beta}$. We can think of the tubular neigbourhood of $\s^1_{\beta}$ as the product of a disk orthogonal to $c$ and an appropriate interval $(-\eps, \eps)$. Therefore, in order to give the coordinates of a point in $\mathcal{U}^{\eps}$, we start by constructing the ones of a point in a unit sphere orthogonal to $c(\theta)$. We consider the following two vectors orthogonal to $c$ and to the tangent vector $\dot{c}$: 
\begin{equation*}
v_1= \left(\begin{array}{c}
0\\
1\\
0.
\end{array}
\right), \qquad 
v_2=\left(\begin{array}{c}
-\sin(\beta) \\
0 \\
\cos(\beta)e^{i\theta}
\end{array}\right)
\end{equation*}
For fixed $\theta$, $v_1$ and $v_2$ span a plane in $\R^4$ which is orthogonal to $c$ at $c(\theta)$. Therefore a point in the unit sphere normal to $c(\theta)$ can be written in the following coordinates:
\begin{equation*}
w= \sin(\varphi) v_1+\cos(\varphi)v_2=\left(\begin{array}{c}
-\sin(\beta)\sin(\varphi) \\
\cos(\varphi) \\
\cos(\beta)\sin(\varphi)e^{i\theta}
\end{array} \right), \quad \mbox{ for } \varphi \in [0,2\pi).
\end{equation*}
As for a point in a tubular neighbourhood of $c(\theta)$, we then get:
\begin{equation*}
F(r,\theta,\varphi)=\cos(r)c(\theta)+\sin(r)w, 
\end{equation*}
where $r$ varies in a small interval $(-\eps,\eps)$ in $(0,\pi)$, for some $\eps$ depending on $\beta$. We are going to write the metric in these new coordinates. Observe that we have: 
\begin{align*}
\partial_r F & =-\sin(r)c(\theta)+\cos(r)w, \quad |\partial_r F|^2= 1. \\
\partial_{\varphi}F &=\sin(r)
\left(
\begin{array}{c}
-\sin(\beta)\cos(\varphi) \\
-\sin(\varphi) \\
\cos(\varphi)\cos(\beta) e^{i\theta}. 
\end{array} 
\right)
, \quad |\partial_{\varphi}F|^2=\sin^2(r). \\
\partial_{\theta}F &=\cos(r) \left(\begin{array}{c}
0 \\
0\\
i\sin(\beta)e^{i\theta}
\end{array} \right)+ \sin(r) \left(\begin{array}{c}
0\\
0\\
i\cos(\beta)\sin(\varphi)e^{i\theta}
\end{array} \right)\\
|\partial_{\theta}F|^2 &=\cos^2(r)\sin^2(\beta)+\sin^2(r)\cos^2(\beta)\sin^2(\varphi)+\cos(r)\sin(r)\sin(2\beta)\sin(\varphi). 
\end{align*}
Note that all the mixed terms vanish. Therefore, in $\mathcal{U}^{\eps}$ the metric can be written in the following form:
\begin{align*}
g_{0}&=dr^2+\sin^2(r)d\varphi^2
+[\cos^2(r)\sin^2(\beta)+\sin^2(r)\cos^2(\beta)\sin^2(\varphi)
\\&+\cos(r)\sin(r)\sin(2\beta)\sin(\varphi)]d\theta^2.
\end{align*}
Observe that when  $\beta=\frac{\pi}{2}$, then $g_{0}$ is the round metric on $\s^3$ written as a doubly warped product on $(0,\pi)\times \s^1\times \s^1$. When we consider $\eps$ and $r$ going to zero, we get the following asymptotic expansion:
\begin{equation*}
g_{0} - (dr^2+r^2d\varphi^2+\sin^2(\beta)d\theta^2) = r\sin(2\beta)\sin(\varphi)d\theta^2+r^2\cos^2(\beta)\sin^2(\varphi)d\theta^2+o(r^3). 
\end{equation*}
When $\beta=\frac{\pi}{2}$ and thus the circle is totally geodesic, the term with factor $r$ vanishes, so that we can estimate the term on the right by some constant $\Lambda$ times $r^2$. When $\beta \neq \frac{\pi}{2}$, we estimate the term on the right by $\Lambda r +o(r)$. This proves the desired inequality \ref{asymptotic}.

\subsection*{Weak gradients} In this part, we prove the following proposition

\begin{proposition} \label{gradient}

Let $(X,g)$ be a $n-$dimensional stratified space and $m$ be the corresponding Riemannian measure. Then, for any compactly supported Lipschitz function $f$, the following equality holds $m$-a.e.

$$  |\nabla f|_{*}= \lip(f)$$
where the term on the l.h.s. is the minimal relaxed gradient   while the term on the other side is the local Lipschitz constant.

\end{proposition}

\begin{proof}

Our argument is mainly based on the fact that the Liouville measure $\Lo$ on the unit bundle of a Riemannian manifold is preserved by the geodesic flow.

Let $f$ be a compactly supported Lipschitz function on $X$ and let $f_n \in L^2 (X,m)$ be a sequence of Lipschitz functions converging to $f$ in $L^2(X,m)$  such that $|D f_n|$ weakly converge to $| Df|_{*}$ in $L^2(X,m)$. Let $B(o,R)$ be a ball on which $f$ is supported, note that by definition of $|Df|_{*}$, we can assume that the functions $f_n$ are supported in $B(o,2R)$ (just replace $f_n$ by $f_n h$ where $h$ is a Lipschitz function such that $1 \geq h\geq 0$, $h$ equals $1$ on  $B(o,R)$ and is supported on $B(o,2R)$, and conclude by minimality of the relaxed gradient). Moreover, up to mollify the $f_n$, we can further assume that $f_n$ are $C^1$ functions on the regular subset $X^{\tiny{reg}}$ of $X$. Since we look for an equality that may fail on a negligible subset, we shall restrict our attention to $X^{\tiny{reg}}$.

 In what follows, the notation  $\fint_A f \,d\mu$  means $\int_A f \,d\mu/ \mu(A)$. Let $z\in X^{\tiny{reg}}$ be a point where $f$ is differentiable and $D_z f$ be the differential of $f$ at $z$, then 
 \begin{equation}\label{average}
  \lip(f)(z)= |D_zf| = \frac{1}{c_n}\lim_{\eta \downarrow 0} \fint_{\S} \frac{|f(\exp_z (\eta u)) -f(z)|}{\eta} du
 \end{equation}
 where $c_n= 2/\big( (n-1) \int_0^{\pi} \sin^{n-2}(s) \,ds\big)$.
 
 Acccording to Lebesgue's theorem, for $m-a.e. \,x \in X$, it holds
$$  \lip(f)(x) = \lim_{r \downarrow 0} \fint_{B(x,r)}  \lip(f)(z) \, dv(z).$$
Thus, using Rademacher's theorem,  we infer from (\ref{average}) that for $m-a.e. \,x \in X$,

 \begin{eqnarray}\label{average2}   \lip(f)(x) &=&\lim_{r \downarrow 0} \fint_{B(x,r)} |D_zf| \, dv(z)  \nonumber\\
 &=& \lim_{r \downarrow 0} \lim_{\eta \downarrow 0} \fint_{B(x,r) \times \S} \frac{1}{c_n}\frac{|f(\exp_z (\eta u)) -f(z)|}{\eta} \, d\Lo (z,u).
  \end{eqnarray}

By combining the invariance of $\Lo$ under the geodesic flow with the $L^2$-convergence of $f_n$ to $f$, we get
\begin{eqnarray*} \lim_{n \rightarrow + \infty} \fint_{B(x,r) \times \S} |f(\exp_z (\eta u))-f_n(\exp_z (\eta u))| \, d\Lo (z,u) = \\
 \lim_{n \rightarrow + \infty} \fint_{B(x,r) \times \S} |f(z)-f_n(z)| \, d\Lo (z,u) 
= 0.
\end{eqnarray*}

The above equality allows us to rewrite (\ref{average2}) as

\begin{eqnarray}\label{truc}
\fint_{B(x,r)} |D_zf| \, dv(z)  & =&  \lim_{\eta \downarrow 0} \lim_{n \rightarrow + \infty} \fint_{B(x,r) \times \S} \frac{1}{c_n}\frac{|f_n(\exp_z (\eta u)) -f_n(z)|}{\eta} \, d\Lo (z,u) \nonumber \\
                                             &\leq &   \lim_{\eta \downarrow 0}\lim_{n \rightarrow + \infty}  \fint_{B(x,r) \times \S}   \frac{1}{c_n} \int_0^1  |D_{\exp_z (s\eta u) }f_n (\textstyle{\frac{\partial}{\partial s}})| \,ds \,d\Lo (z,u) \nonumber\\
                                              &\leq &   \lim_{\eta \downarrow 0}\lim_{n \rightarrow + \infty} \int_0^1 \fint_{B(x,r) \times \S}   \frac{1}{c_n}  |D_{z}f_n( u)| \,d\Lo (z,u)ds\nonumber\\
                                              &\leq &   \lim_{\eta \downarrow 0}\lim_{n \rightarrow + \infty}  \fint_{B(x,r) }  |D_zf_n|  \,dv(z)\nonumber\\
                                              &\leq &   \lim_{\eta \downarrow 0}  \fint_{B(x,r)}  | \nabla f|_{*}(z)  \,dv(z)= \fint_{B(x,r)}  | \nabla f|_{*}(z)  \,dv(z)
\end{eqnarray}
where the second inequality follows again from the invariance of $\Lo$ w.r.t. the geodesic flow and the last one from the weak convergence of $Lip(f_n)$ to $|\nabla f|_{*}$ in $L^2(X,m)$. To conclude, we combine (\ref{average2}) and (\ref{truc}) which gives the result thanks to Lebesgue's theorem.



\end{proof}